\documentclass{article}
\usepackage[utf8]{inputenc}
\usepackage[english]{babel}
\usepackage{amsmath}
\usepackage{amssymb}
\usepackage{amsthm}
\usepackage{enumerate}
\usepackage{color,graphicx}
\usepackage{lastpage209}

\newtheorem{lemma}{Lemma}

\newtheorem{proposition}{Proposition}
\theoremstyle{definition}
\newtheorem{definition}{Definition}
\newtheorem{remark}{Remark}

\newtheorem{example}{Example}

\def\NN{{\mathbb N}}

\def\L{\mathcal L}

\def\J{\mathcal J}
\def\RR{\mathcal R}
\def\L{\mathcal L}

\def\mod{{\rm mod}\,}

\begin{document}

\title{Effects of semigroup properties on local embeddability}

\author{Dmitry Kudryavtsev\footnote{Department of Mathematics, University of York, Deramore Lane, York, YO10 5GH, United Kingdom. Email dmitry.kudryavtsev@york.ac.uk}}

\maketitle

\begin{abstract}

We investigate whether semigroups with a given property which are also locally embeddable into finite semigroups can be locally embedded into finite semigroups with the same property, obtaining a positive answer for completely simple and Clifford semigroups (similarly to group and inverse semigroup cases studied previously) and a negative answer for $\J$-trivial semigroups (similarly to cancellative semigroups). Additionally, we resolve the standing question on the differences between local embeddability in and local wrappability by finite structures, providing a novel construction of $\J$-trivial semigroups which satisfy the latter, but not the former.

finite semigroups, local embeddability into finite, $\J$-trivial semigroups


\end{abstract}

\section*{Acknowledgments}

The work was financially supported by the London Mathematical Society Early Career Fellowship.

The author is grateful to Prof Mark Kambites and Prof Victoria Gould for the fruitful discussions about semigroups in general and finitary conditions in particular.

\section*{Introduction}

Understanding the intrinsic structural differences between finite and infinite structures defined by a fixed set of algebraic properties is a broad and important area of mathematics. It includes venerable topics such as the resolution of the Burnside problem on the finiteness of the group in which every element has a finite order obtained by Golod and Shafarevich with further results for the shared exponent obtained by Adian and Novikov in 1960s (see \cite{Adi} and references therein), as well as more recent developments such as the investigation of the notion of soficity for groups (see early 2000s works of Gromov \cite{Gro} and Weiss \cite{Wei}) and semigroups (see 2010s works of Ceccherini-Silberstein and Coornaert \cite{CSC} and Kambites \cite{Kam}). 

Similarly to soficity, the concept at the crux of the present text, local embeddability into finite semigroups (LEF for short), is a well-established finitary condition (that is, one shared by all finite structures), which describes all structures internally similar to or approximable by the finite ones. It was initially examined in \cite{Ev} with relation to residual finiteness and word problems, reintroduced in \cite{GV} specifically for groups, studied for general structures in \cite{Bel} via model theory and topology, and recently brought into the class of semigroups in \cite{K23, K24}.

The aim of this paper is twofold: we investigate the interactions between LEF and classical semigroup properties, namely whether being LEF with an extra condition X is equivalent to being locally embeddable into the class of finite semigroups satisfying the condition X, and also expand on the local wrappability by finite semigroups (LWF for short), a generalisation of LEF introduced in \cite{K24}. This is reflected in the structure of the work. Section 1 covers the basics of LEF, LWF and important classical semigroup notions. Section 2 deals with the first goal, providing positive answers for Clifford semigroups and completely simple semigroups as well as a negative one for $\J$-trivial semigroups. Finally, in Section 3 we demonstrate that there exist non-LEF semigroups which are LWF, unlike the group case. Appendices A and B contain supplementary tables used in the $\J$-trivial example in Section 2.

\section{Preliminaries}

We begin with a definition of the main property studied in this paper.

\begin{definition}\label{def_lef}
A (semi)group $S$ is called {\em locally embeddable into the class of finite (semi)groups} (an {\em LEF} (semi)group  for short) if for every finite subset $H$ of $S$ there exist a finite (semi)group $F_H$ and an injective function $f_H: H \rightarrow F_H$ such that for all  $x,y \in H$ with $xy \in H$ we have $(xy)f_H = (xf_H)(yf_H)$.

We call $(F_H, f_H)$ an {\em approximating pair} for $H$.
\end{definition}

The following concept underlines that Definition \ref{def_lef} is local and can be failed at a specific ``point" within a semigroup.

\begin{definition}\label{def_nemb}
We say that a finite set $H$ with partially defined multiplication which satisfies associativity for any valid inputs is {\em non-embeddable} if there do not exist a finite (semi)group $F_H$ and an injective function $f_H: H \rightarrow F_H$ such that for all  $x,y \in H$ with $xy \in H$ we have $(xy)f_H = (xf_H)(yf_H)$.
\end{definition}

We can also immediately observe the basic substructure inheritance of being LEF.

\begin{proposition}\label{prop_sublef}\emph{\cite[Proposition 5.1]{K23}}
Let $S$ be an LEF (respectively, residually ﬁnite) semigroup and $T$ a subsemigroup of $S$. Then $T$ is an LEF (respectively, residually ﬁnite) semigroup.
\end{proposition}

It has been shown previously that the free groups and semigroups are LEF, while the bicyclic monoid $B = \langle a,b \mid ab=1\rangle$ is not LEF and the set $\{1,a,b,ba\}$ with the partial multiplication inherited from $B$ is non-embeddable (see \cite[Examples 1.5 and 1.7]{K23} for the semigroup and monoid results).

Another property examined previously and inspired by \cite{GV} is as follows.

\begin{definition}\label{def_lwf}{\cite[Definition 3]{K24}}
We say that a (semi)group $S$ is {\em locally wrapped by the class of finite (semi)groups} (an {\em LWF} (semi)group for short) if for every finite subset $H$ of $S$ there exist a finite (semi)group $D_H$ and a function $d_H: D_H \rightarrow S$ such that $H \subseteq D_H d_H$ and for all $x', y' \in D_H$ with $x'd_H,y' d_H \in H$ we have $(x'y')d_H = (x'd_H)(y'd_H)$.
\end{definition}

\begin{proposition}\label{prop_lefislwf}\emph{\cite[Proposition 3]{K24}}
Let $S$ be an LEF (semi)group. Then it is LWF.
\end{proposition}

In the group case, the converse is also true, as established in \cite[Corollary 25]{K24}.

We also present several classical semigroup properties for later examination. Let us start with the following.

\begin{definition}\label{def_jtriv}
Let $S$ be a semigroup and $x,y$ elements of $S$.

By $S^1$ we understand either $S$ if it contains an identity or the extension of $S$ by an external identity otherwise. The pre-orders $\le_\L, \le_\RR, \le_\J$ on $S$ and the corresponding {\em Green's relations}, denoted $\L, \RR, \J$, are defined as follows:
\begin{itemize}
\item $x \le_\L y$ if $S^1  x\subseteq S^1 y$, or equivalently there exists $z \in S^1$ with $x = zy$;
\item $x \le_\RR y$ if $x S^1 \subseteq y S^1$, or equivalently there exists $z \in S^1$ with $x = yz$;
\item $x \le_\J y$ if $S^1  x S^1 \subseteq S^1 y S^1$, or equivalently there exist $z,u \in S^1$ with $x = uyz$;
\item $x \mathrel{\L} y$ if $x \le_\L y$ and $y \le_\L x$;
\item  $x \mathrel{\RR} y$ if $x \le_\RR y$ and $y \le_\RR x$;
\item $x \mathrel{\J} y$ if $x \le_\J y$ and $y \le_\J x$.
\end{itemize}

The semigroup $S$ is called $\J$-{\em trivial} if $s \mathrel{\J} t$ implies $s = t$ for any $s,t\in S$. Being {\em $\RR$-trivia}l and {\em $\L$-trivial} is defined similarly.
\end{definition}

Another point of interest is the family of semigroup classes described below. We begin with the most straightforward one.

\begin{definition}\label{def_compsimp}
A semigroup $S$ is called {\em completely simple} if 
\begin{itemize}
\item $S$ does not contain any non-trivial ideals, that is there is no set $A \subsetneq S$ such that $AS \subseteq A, SA \subseteq A$;
\item $S$ contains a minimal $\L$-class and a minimal $\RR$-class with regards to $\le_\L$ and $\le_\RR$ respectively.
\end{itemize}
\end{definition}

These semigroups can be characterised using another classical construction.

\begin{definition}\label{def_ReesM}
Let $G$ be a group, $I$ and $\Lambda$  sets and $P$ a matrix indexed by $\Lambda \times I$ with elements from $G$. The {\em Rees matrix semigroup} defined by this tuple, denoted $M(G; I, \Lambda; P)$, is equal to $I \times G \times \Lambda$ setwise, with the product defined as $$(i,g,\lambda) \cdot (j,h, \mu) = (i, g P_{\lambda j} h, \mu),$$
where $i, j \in I, \lambda, \mu \in \Lambda, g,h \in G$ and $P_{\lambda j} \in G$ is the element of $P$ at the position $\lambda, j$.
\end{definition}

It is well-known (see, for example, \cite[Theorem 3.3.1]{How95}), that every Rees matrix semigroup is completely simple, and for every completely simple semigroup there exists an isomorphic Rees matrix semigroup.

To understand the second class, we require an auxiliary definition.

\begin{definition}
Let $(E, \le)$ be a meet semilattice. Also, let $\{S_e\}_{e \in E}$ be a family of disjoint semigroups indexed by elements of $E$. Furthermore, let it be equipped with homomorphisms $\phi_{e_1,e_2}: S_{e_1}\rightarrow S_{e_2}$ for $e_1 \ge e_2$ in the natural partial order satisfying the following properties:

\begin{itemize}
\item $\phi_{e,e}$ is identity on $S_e$ for every $e \in E$;
\item $\phi_{e_1,e_2} \phi_{e_2,e_3} = \phi_{e_1,e_3}$ for all $e_1 \ge e_2 \ge e_3$. 
\end{itemize}

We say that the union $\bigsqcup\limits_{e\in E} S_e$ is a {\em semilattice of semigroups}, which is itself a semigroup with the operation $x \cdot y = (x \phi_{e_1,e \wedge e_2})(y \phi_{e_2,e_1 \wedge e_2})$ for $x \in S_{e_1}$, $y \in S_{e_2}$. 
\end{definition}

\begin{definition}
A semigroup $S$ is called {\em Clifford} if it is a semilattice of groups.
\end{definition}

While other characterisations of Clifford semigroups exist, for example via inverse semigroups, we require this particular approach to demonstrate our results.

The classes of completely simple and Clifford semigroups are contained within another important overarching class.

\begin{definition}
A semigroup $S$ is called {\em completely regular} if every element of $S$ belongs to a subgroup of $S$.
\end{definition}

These objects also possess a developed structural theory (see \cite{Pet}).

With the key concepts established, we can proceed with discussing the problems outlined in the introduction. 

\section{LEF and semigroup properties}

Let X denote some kind of semigroup property (for instance, being a group, cancellative, completely simple or $\J$-trivial). We can consider a semigroup $S$ which is X and simultaneously approximable by finite semigroups, and ask a natural question on whether $S$ is approximable by finite semigroups which are also X? The following definition turns this idea formal for the case of local embeddability.

\begin{definition}
Let $\mathcal{K}$ be a class of semigroups and $S$ be a semigroup in $\mathcal{K}$. It is called {\em locally embeddable into the class of finite $\mathcal{K}$-semigroups} if for every finite subset $H$ of $S$ there exist a finite semigroup $F_H \in \mathcal{K}$ and an injective function $f_H: H \rightarrow F_H$ such that for all  $x,y \in H$ with $xy \in H$ we have $(xy)f_H = (xf_H)(yf_H)$.

As with general LEF, we call $(F_H, f_H)$ an {\em approximating pair} for $H$.
\end{definition}

This gives us a general type of mathematical problem: if $\mathcal{K}$ is a class of semigroups and $S$ is an LEF semigroup which is in $\mathcal{K}$, is it true that $S$ is  locally embeddable into the class of finite $\mathcal{K}$-semigroups? The previous studies into LEF already provide us some examples when the answer is positive.

\begin{proposition}\label{prop_lefgr}\emph{\cite[Proposition 2.2]{K23}}
A group is an LEF semigroup if and only if it is an LEF group.
\end{proposition}

\begin{definition}\label{def_ilef}{\cite[Definition 7]{K24}}
A (semi)group $S$ is called {\em locally embeddable into the class of finite inverse semigroups} (an {\em iLEF} semigroup  for short) if for every finite subset $H$ of $S$ there exist a finite (semi)group $F^{(i)}_H$ and an injective function $f^{(i)}_H: H \rightarrow F^{(i)}_H$ such that for all  $x,y \in H$ with $xy \in H$ we have $(xy)f^{(i)}_H = (xf^{(i)}_H)(yf^{(i)}_H)$.
\end{definition}

\begin{proposition}\label{prop_lefgr}\emph{\cite[Proposition 18]{K24}}
An inverse semigroup is an LEF semigroup if and only if it is an iLEF semigroup.
\end{proposition}

However, there are also cases when the answer is negative.

\begin{example}
Consider a famous construction by A. Malcev, the cancellative semigroup $C = Sg \langle a,b,c,d,x,y,u,v| ax = by, cx = dy, au =bv\rangle$ which is not embeddable (not locally, but generally) into groups.

Note that $C$ is LEF. In order to see this, choose an arbitrary finite subset $H$ of $C$. Since for any given element of $C$ the words in $\{a,b,c,d,x,y,u,v\}$ representing it must have the same lengths due to the fact that all the defining relations identify terms of equal length, there exists $n$ greater than the length of representatives for any $h \in H$. The set $I \subseteq C$ of elements represented by words in $\{a,b,c,d,x,y,u,v\}$ of length $n$ or above forms an ideal, the quotient semigroup $C /I$ is finite as the alphabet is finite, and the natural homomorphism $\pi: C \rightarrow C/I$ is injective on $H$, meaning that $(C/I, \pi|_H)$ is an approximating pair for $H$.

However, $C$ is not locally embeddable into finite cancellative semigroup (which are the same as finite groups). Consider the finite subset $\{a,b,c,d,x,y,u,v,ax,cx,au,cu,dv\}$ and assume that $(F, f)$ is an approximating pair for it, where $F$ is a group. Since we have $(af)(xf) = (bf)(yf), (cf)(xf) = (df)(yf), (af)(uf) =(bf)(vf)$, we also have $(df)^{-1} (cf) = (yf)(xf)^{-1} =  (bf)^{-1} (af) = (vf)(uf)^{-1} $. However, from  $(df)^{-1} (cf)  = (vf)(uf)^{-1} $ it follows that  $(cf) (uf)  = (df)(vf)$, which gives us $(cu) f = (dv) f$ by the multiplicative property, which contradicts the injectivity of $f$ on the chosen subset as $cu \neq dv$ in $C$.
\end{example}

Now we can turn to the analysis for the classes of completely simple, Clifford and $\J$-trivial semigroups. The first among these requires the following additional statements.

\begin{lemma}\label{lem_compsimp}
Let $M(G; I, \Lambda; P)$ be a Rees matrix semigroup and assume that $G$ is LEF. Then for any finite subset $H'$ of $M(G; I, \Lambda; P)$ there exists an approximating pair $(F, f)$ such that $F$ is a Rees matrix semigroup. 
\end{lemma}
\begin{proof}
Since the set $H'$ is finite, the sets $J = \{ i \in I | \exists (i,g,\lambda) \in H' \}$, $\Sigma = \{\lambda \in \Lambda | \exists (i,g,\lambda) \in H' \}$ and $X = \{g \in G | \exists (i,g,\lambda) \in H' \}$ are also finite. Furthermore, $Y = \{P_{\lambda i} | \lambda \in \Sigma, i \in J\}$ is also finite.

Let $F'$ be a finite group and $f'$ be a map from $X$ to $F'$ such that $(F',f')$ is an approximating pair for $X \cup Y \cup (X \cup Y)^2 \cup (X \cup Y)^3$. We claim that $F = M(F'; J, \Sigma; P')$ together with the map $f: H' \rightarrow F$ defined by $(i,g,\lambda) f  = (i, g f' ,\lambda)$ is an approximating pair for $H'$, where $P'$ is the matrix obtained by taking $\Sigma, J$-indexed submatrix from $P$ and applying $f'$ to all of the entries.

First, note that $F$ is finite as all the components are finite and setwise $F$ is equal to $J \times F' \times \Sigma$. Second, the map $f$ is injective on $H'$ because distinct elements of $H'$ differ either by index on the left or right, or by the group element, and these differences are explicitly preserved by the construction as $f'$ is also injective on $X \subseteq X \cup Y$. Finally, if $(i,g,\lambda), (j, h, \mu)$ are elements of $H'$, we have 
\begin{align*}
((i,g,\lambda) \cdot (j, h, \mu))f  & =  (i,g P_{\lambda j} h,\mu) f = (i,(g P_{\lambda j} h) f',\mu) = (i,gf' P_{\lambda j} f' h f',\mu)  \\
 & =  (i,g f',\lambda) \odot (j, h f', \mu) = ((i,g ,\lambda)f ) \odot ((j, h , \mu) f),
\end{align*}
where $\cdot$ is the product in $M(G; I, \Lambda; P)$ and $\odot$ is the product in $F$.
\end{proof}

\begin{proposition}\label{prop_compsimp1}
Let $S$ be a completely simple semigroup and $M(G; I, \Lambda; P)$ a Rees matrix semigroup isomorphic to $S$. Then $S$ is LEF if and only if $G$ is LEF.
\end{proposition}
\begin{proof}
If $S$ is LEF, then  $M(G; I, \Lambda; P)$ is also LEF. The latter contains a subgroup isomorphic to $G$, for instance, $G'$ which is formed by all elements $(i,g,\lambda)$ with fixed $i$ and $\lambda$ with the isomorphism $\phi: G \rightarrow G'$ defined by $g \phi = (i, P_{\lambda i}^{-1} g, \lambda)$. Thus, by Proposition \ref{prop_sublef} the group $G$ is also LEF.

Assume that $G$ is LEF. Consider a finite subset $H$ of $S$. To prove that there exists an approximating pair for $H$ it is enough to demonstrate that there exists an approximating pair for the corresponding set $H'$ in $M(G; I, \Lambda; P)$. This holds by Lemma \ref{lem_compsimp}, thus $S$ is also LEF.
\end{proof}

\begin{proposition}
A completely simple semigroup $S$ is LEF if and only if it is locally embeddable into the class of finite completely simple semigroups. 
\end{proposition}
\begin{proof}
If $S$ is LEF, then we can consider a Rees matrix semigroup $M(G; I, \Lambda; P)$ isomorphic to $S$. We have that every finite subset $H$ of $S$ can be approximated by the same semigroup as the corresponding subset $H'$ of $M(G; I, \Lambda; P)$, which we can choose to be completely simple by Lemma \ref{lem_compsimp} as $G$ must be also LEF by Proposition \ref{prop_compsimp1}.

If $S$ is  locally embeddable into finite completely simple semigroups, it is evidently LEF.
\end{proof}

We can also demonstrate that the answer to the general question is positive for Clifford semigroups. In order to do this, we delve into the structure of lattices of semigroups.

\begin{lemma}
Let $S$ and $T$ be LEF semigroups and $\phi$ a homomorphism from $S$ to $T$. Then for any finite non-empty $H \subseteq S$ there exist approximating pairs $(F, f)$ for $H$ and $(F', f')$ for $H \phi$ such that there is a homomorphism $\psi$ from $F$ to $F'$ which satisfies $(hf)\psi = (h\phi) f'$ for all $h \in H$.
\end{lemma}
\begin{proof}
Let $(F_0, f_0)$ be an arbitrary approximating pair for $H$ and $(F', f')$ an arbitrary approximating pair for $H \phi$. Set $F = F_0 \times F'$ and $f = f_0 \times (\phi f')$, and set $\psi$ to be the projection from $F$ onto the second term of the direct product.

Note that $(F, f)$ is an approximating pair for $H$. To see this, consider $x,y \in H$ such that $xy \in H$ as well. We have 
\begin{align*} (x f) (y f) = & (x f_0 , x  (\phi f') ) \cdot  (y f_0 , y  (\phi f') ) = ((xf_0 )(yf_0), (x  (\phi f') )(y  (\phi f')) = ((xy)f_0, (x\phi \cdot y\phi)f')  \\ = & ((xy)f_0, (x y\phi)f') = (xy) f.
\end{align*}

Furthermore, from the construction we have $(hf)\psi = (h\phi) f'$ for all $h \in H$.
\end{proof}

This technique can be extended for semilattices of semigroups.

\begin{lemma}\label{lem_semlef}
Let $(E, \le)$ be a meet semilattice. Also, let $S = \bigsqcup\limits_{e\in E} S_e$ be a semilattice of LEF semigroups equipped with homomorphisms $\phi_{e_1,e_2}: S_{e_1}\rightarrow S_{e_2}$ for $e_1 \ge e_2$ and $H$ a finite non-empty subset of $S$. Denote by $E'$ the finite subsemilattice of $E$ generated by the finite set consisting of $e$ such that $S_e \cap H \neq \emptyset$. Furthermore, let $K$ be $$\bigcup\limits_{e_1, e_2 \in E', e_1 \ge e_2}  (H \cap S_{e_1})\phi_{e_1,e_2},$$

and let $(T_e, t_e)$ be arbitrary approximating pairs for $K \cap S_e$ for every $e \in E'$.

Finally, let $F_e = \prod\limits_{e' \in E', e' \le e} T_{e'}$ and $f_e$ to be the restriction of $f'_e = \prod\limits_{e' \in E', e' \le e} \phi_{e,e'} t_{e'} $ on $H \cap S_e$. For $e_1, e_2 \in E'$ with $e_1 \le e_2$ define $\psi_{e_1,e_2}: F_{e_1}\rightarrow F_{e_2}$ to be the natural projection from the former to the latter.

Then $F = \bigsqcup\limits_{e\in E'} F_e$ is a semilattice of finite semigroups and together with the map $f = \bigsqcup\limits_{e\in E'} f_e$ it forms an approximating pair for $H$.

\end{lemma}
\begin{proof}
To start, we need to check that $F$ is a semilattice of semigroups, that is 

\begin{itemize}
\item $\psi_{e,e}$ is identity on $F_e$ for every $e \in E'$;
\item $\psi_{e_1,e_2} \psi_{e_2,e_3} = \psi_{e_1,e_3}$  for all $e_1 \ge e_2 \ge e_3$.
\end{itemize}

Both these properties follow immediately from the construction of $\psi$'s as projections. Additionally. for $e_1,e_2 \in E'$  such that $e_1 \ge  e_2$ and $h \in H \cap S_{e_1}$ we have 
$$(h f'_{e_1})\psi_{e_1,e_2} =  (h \phi_{e_1,e_2}) f'_{e_2}$$ as $h f'_{e_1}$ is a tuple consisting of $h \phi_{e_1,e'} t_{e'}$ for $e_1 \ge e'$,  $(h \phi_{e_1,e_2}) f'_{e_2}$ is a tuple consisting of $(h \phi_{e_1,e_2}) \phi_{e_2,e'} t_{e'}$ for $e_2 \ge e'$ and  $\phi_{e_1,e_2} \phi_{e_2,e'} = \phi_{e_1,e'}$.

Note that $(F_e, f_e)$ is an approximating pair for $H\cap S_e$. First, $F_e$ is finite as a finite union of finite semigroups. Second, $f_e$ is injective because one of the factors of $F_e$ is $T_e$, which approximates $K \cap S_e \supseteq H\cap S_e$, meaning that all the elements will be separated in this factor. In order to see the multiplicative property, we will prove a more general statement, that is if $x \in H \cap S_{e_1}$, $y \in H \cap S_{e_2}$ and $xy \in H \cap S_{e_3}$ where $e_3 = e_1 \wedge e_2$, then $(x f) (yf) = (x f_{e_1}) (y f_{e_2}) = (xy) f_{e_3} = (xy) f$.

The first and the last of the equations above hold by the construction. Thus, we need to demonstrate that $(x f_{e_1}) (y f_{e_2}) = (xy) f_{e_3}$. We have
\begin{align*}
(x f_{e_1}) (y f_{e_2}) & =  (x  f_{e_1} \psi_{e_1,e_3}) (y f_{e_2} \psi_{e_2,e_3}) =  (x  f'_{e_1} \psi_{e_1,e_3}) (y f'_{e_2} \psi_{e_2,e_3}) \\
& = ((x \phi_{e_1,e_3})  f'_{e_3})  ((y  \phi_{e_2,e_3})f'_{e_3}) = ((x \phi_{e_1,e_3}) (y  \phi_{e_2,e_3})) f'_{e_3} = (xy) f'_{e_3} = (xy) f_{e_3}.
\end{align*}

This also guarantees the multiplicative property for $H$ and $(F,f)$. To end the proof, note that $F$ is finite as a finite semilattice of finite groups, and that $f$ is injective because elements of $H$ either fall in different $F_e$'s or separated within the given $F_e$ by the observation above.
\end{proof}

\begin{proposition}\label{prop_cliff1}
Let $S$ be a Clifford semigroup and $Z$ be a semilattice of groups $\bigsqcup\limits_{e\in E} G_e$ isomorphic to $S$. Then $S$ is LEF if and only if every $G_e$ is LEF.
\end{proposition}
\begin{proof}
If $S$ is LEF, then $Z$ is LEF and every $G_e$ is LEF as its subsemigroup by Proposition \ref{prop_sublef}.

Assume that every $G_e$ is LEF.  Consider a finite subset $H$ of $S$. To prove that there exists an approximating pair for $H$ it is enough to demonstrate that there exists an approximating pair for the corresponding set $H'$ in $Z$. This holds by Lemma \ref{lem_semlef}, meaning that $S$ is also LEF.
\end{proof}

\begin{proposition}
A Clifford semigroup $S$ is LEF if and only if it is locally embeddable into the class of finite Clifford semigroups. 
\end{proposition}
\begin{proof}
If $S$ is LEF, then we can consider a semilattice of groups $Z$ isomorphic to $S$. We have that every finite subset $H$ of $S$ can be approximated by the same semigroup as the corresponding subset $H'$ of $Z$, which we can choose to be a semilattice of finite groups, i.e. a finite Clifford semigroup, by setting the approximations $T_e$ in Lemma \ref{lem_compsimp} to be groups, which is possible by Propositions \ref{prop_lefgr} and \ref{prop_cliff1}.

If $S$ is locally embeddable into finite Clifford semigroups, it is evidently LEF.
\end{proof}

It is worth noting that the question remains open for the class of completely regular semigroups. 

\begin{remark}The difficulty in this case arises from the fact that while they can be thought of as collections of completely simple semigroups, which are relatively well-understood in general as well as in terms of being LEF, the connections between the constituent parts are more complicated than for Clifford semigroups. Namely, the completely simple semigroups are also organised into a semilattice, but the maps go from the ones above into the translational hull of the ones below, and these images can hide an additional degree of non-embeddability.
\end{remark}

The remainder of the section is devoted to producing an example of a $\J$-trivial LEF semigroup which is not locally embeddable into finite $\J$-trivial semigroups.

Consider $Q = Sg \langle a,b,c,e,x | xb = cx, ac = ca, ae=ea, ec=ce, xca = xe, aex = ax\rangle$.

\begin{proposition}
The semigroup  $Q = Sg \langle a,b,c,e,x | xb = cx, ac = ca, ae=ea, ec=ce, xca = xe, aex = ax\rangle$ is $\J$-trivial.
\end{proposition}
\begin{proof}
Assume that there exist two elements $s,t$ of $Q$ such that $s \neq t$ and $s$ is $\J$-related to $t$. Let $u$ and $v$ be words in $\{a,b,c,e,x\}$ representing $s$ and $t$. By the choice of $s$ and $t$ there exist possibly empty words $l_1, r_1$ and $l_2, r_2$ such that $l_1 u r_1$ represents $t$ and $l_2 v r_2$ represents $s$. In particular, this means that we can rewrite $v$ into $l_1 u r_1$ using the defining relations and similarly we can rewrite $u$ into $l_2 v r_2$. Combining these two rewritings we have a finite sequence of words $w_0 = u, \ldots, w_n = l_2 l_1 u r_1 r_2$ such that neighbouring words are different only by a single application of the defining relations of $Q$. Note that none of the rewritings affect the existence of $x$'s inside of words, which means that $l_1, l_2, r_1, r_2$ must be free of $x$'s. There are two further possibilities.

If $u$ does not contain any $x$'s itself, then the only applicable relations are $ea = ae$, $ec=ce$ and $ac = ca$ which do not affect the length of the word, which means that $l_1,l_2, r_1, r_2$ must be empty and $u,v$ represent the same element of $Q$ which contradicts our choice of these words.

Otherwise, $u$ has the form $m_0 x m_1 x \cdots x m_k$ where $k \ge 1$ and $m_i$'s are possibly empty words in $\{a,b,c,e\}$. Denote by $u_i$ the prefix of $w_i$ up to the first occurrence of $x$ (in particular, $u_0 = m_0 x$ and $u_n = l_2 l_1 m_0 x$). Our rewritings affect $u_i$'s by being applied directly to them, thus having the form $xb = cx, ac = ca, ae=ea, ec=ce$ or  $aex = ax$. Note that none of these rewritings add or remove $a$'s from the word, which means that $u_n$ and $u_0$ have the same number of $a$'s in them.

Now we can apply similar analysis to the suffixes $v_i$ of $w_i$ starting with the last occurrence of $x$ (in particular, $v_0 =x m_k$ and $v_n = x m_k r_1 r_2$). Similarly, our rewritings affecting $v_i$'s must be of the form $xb = cx, ac = ca, ae=ea, ec=ce$ or $xca = xe$. Note that these rewritings keep  the sum $a_i + e_i$ constant where $a_i$ and $e_i$ are the number of occurrences of the respective letter in $v_i$. In particular, this means that $r_1$ and $r_2$ cannot contain $a$ or $e$ as otherwise $a_n + e_n > a_0 + e_0$. 

Bringing the facts about the prefixes and suffixes together we can conclude that $r_1$ and $r_2$ are empty words and $l_1, l_2$ are simply powers of $e$ since the defining relations of $Q$ keep constant the difference between the number of occurrences of $a$'s and the sum of occurrences of $b$'s and $c$'s (meaning $l_1,l_2,r_1,r_2$ do not contain $b$'s or $c$'s as they do not contain any $a$'s). Write $l_2 l_1 = e^t$.

If $t=0$, $u$ and $v$ represent the same element of $Q$ which contradicts our choice of these words. If $t = 1$, then either $l_1$ or $l_2$ is empty and $u,v$ represent the same element of $Q$ which contradicts our choice of these words. If $t > 1$, $e^t m_0 x m_1 x \cdots x m_k = m_0 x m_1 x \cdots x m_k$, we claim that $e m_0 x m_1 x \cdots x m_k = m_0 x m_1 x \cdots x m_k$ as well.

This follows from the fact that if our rewritings allow us to eliminate a non-zero number of $e$'s from the prefix $u_n = e^t m_0 x$, then they allow us to eliminate a single $e$ as well. To see this, consider the possible structures of $m_0$.

Case 1. $m_0$ does not contain $a$'s. In this case it is impossible to eliminate any $e$'s as the only rule to do so without $x$ on the left is $aex = ax$.

Case 2. $m_0$ has any $b$'s. It is impossible to remove them when there are no $x$'s to the left of them, and it is also impossible to eliminate any $e$'s as we cannot push them past these $b$'s.

Case 3. $m_0$ contains an $a$ and it does not contain any $b$'s. By using the commutativity of $a,c$ and $e$ as well as the relation $aex = ax$ we can eliminate any number of $e$'s. 

To sum up, we have that $e u$ and $u$ represent the same element of $Q$. Since $v$  and $l_1 u$ represent the same element of $Q$, and $l_1$ is a power of $e$, this means that $u$ and $v$ represent the same element of $Q$ which contradicts our choice of these words.

Thus, all cases are impossible and the initial assumption is incorrect.
\end{proof}

\begin{proposition}
The semigroup  $Q = Sg \langle a,b,c,e,x | xb = cx, ac = ca, ae=ea, ec=ce, xca = xe, aex = ax\rangle$ cannot be locally embedded into the class of finite $\J$-trivial semigroups.
\end{proposition}
\begin{proof}
Let $H$ be the subset of $Q$ consisting of all elements represented by words of length $3$ or less in the alphabet $\{a,b,c,e,x\}$.

Assume that there exists $F_H, f_H$ satisfying the requirements of Definition \ref{def_lef} for $H$ with $F_H$ being $\J$-trivial. Denote $a' = a f_H, b' =b f_H, c' = c f_H, e' = ef_H$ and $x' = x f_H$. By the multiplication laws we get $x'b' = c'x', a'c' = c'a', a'e'=e'a',c'e' =e'c', x'c'a' = x'e', a'e'x' = a'x'$.

Since $F_H$ is finite and $\J$-trivial, there exists $n \in \NN$ such that $a^{\prime n}$ is an idempotent and  furthermore $a^{\prime n}= a^{\prime n+1}$. This implies that $x' a^{\prime n} x' = x' a^{\prime n+1} x'$. Note that $x' a' x' b' = x' a' c' x' = x' c' a' x' = x' e' x'$ and similarly 
$$ x' a^{\prime k +1} x' b' = x' a^{\prime k +1} c' x' = x' c' a^{\prime k +1} x' = x' e' a^{\prime k} x'
= x' a^{\prime k} e' x' = x' a^{\prime k} x'$$ for any $k \ge 1$. Thus, we can get $x' a' x' = x' e' x'$ by multiplying the equality above by $b^{\prime n}$, which contradicts the injectivity of $f_H$ as $xax \neq xex$ in $Q$. The latter follows from the fact that the rewritings corresponding to the generating relations of $Q$ do not change the difference between the number of occurrences of $a$ and the sum of occurrences of $b$ and $c$. This means that our initial assumption is incorrect.
\end{proof}

\begin{remark}
A verbatim proof can be used to demonstrate that $Q$ is not locally embeddable into the classes of $\mathcal{L}$- or $\mathcal{R}$-trivial finite semigroups.
\end{remark}

\begin{proposition}
The semigroup $Q = Sg \langle a,b,c,e,x | xb = cx, ac = ca, ae=ea, ec=ce, xca = xe, aex = ax\rangle$ is LEF.
\end{proposition}
\begin{proof}
Denote by $H_n$ the subset of $Q$ consisting of all elements represented by words of length $n$ or less in the alphabet $\{a,b,c,e,x\}$, $n \ge 5$.

Let $w$ be any word in $\{a,b,c,e,x\}$. It has the form  $m_0 y_0 m_1 y_1 \cdots m_{k-1} y_{k-1} m_k$ where $k \ge 0$, $y_i$'s are the maximal subwords of the form $x$, $b^q$ or $x b^q$ for $q > 0$ and $m_i$'s are words in $\{a,c,e\}$. Denote $k$ by $s(w)$ and denote the generating relations in $Q$ by $\rho_0$.

Consider the following rewriting system.

\begin{enumerate}
\item $xb \rightarrow cx, ca \rightarrow ac, ea \rightarrow ae, ec \rightarrow ce$;

\item $x a^\alpha c^\beta \rightarrow x a^{\alpha - \beta} e^\beta$, $0< \beta \le \alpha$;
\item $x a^\alpha c^\beta \rightarrow x c^{\beta - \alpha} e^\alpha$, $0< \alpha < \beta$;

\item $ a^\alpha c^\beta e^\gamma x  \rightarrow a^\alpha c^\beta x $, $0<\alpha$, $0 \le \beta$, $0 < \gamma$;

\item $x e^\gamma x \rightarrow xex$, $1< \gamma$;
\item $x c^\beta e^\gamma x \rightarrow x c^{\beta} e x$, $0< \beta$, $1 <\gamma$;
\item $x a^\alpha c^\beta e^\gamma x \rightarrow x a^{\alpha - \beta} x$, $0\le \beta < \alpha$, $0\le \gamma$ with $\beta$ and $\gamma$ being not simultaneously zero;
\item $x a^\alpha c^\beta e^\gamma x\rightarrow x e x$, $0 <\alpha = \beta$, $0\le \gamma$;
\item $x a^\alpha c^\beta e^\gamma x\rightarrow x c^{\beta - \alpha} e x$, $0 <\alpha < \beta$, $0\le \gamma$;
\end{enumerate}

It is evident that this system is noetherian as every rule makes the rewritten word smaller in the lexicographical order induced by $a<c<e<b<x$. It is also locally confluent. To see this, consider a word $w$ such that we can apply two rules to its subwords, resulting in distinct words $w_1$ and $w_2$. Our goal is to demonstrate that there exists a word $w_0$ such that we can rewrite both $w_1$ and $w_2$ into $w_0$. If the rewritten subwords do not intersect, the statement is evident. Below we will consider $w = w' t w''$ where $t$ is the joint of the two intersecting subwords undergoing the rewritings and $w', w''$ are arbitrary. Correspondingly, $w_i = w' t_i w''$, $i \in \{0,1,2\}$. Note that for all rewritings where the intersection is the letter $x$ and neither of the rules is $xb \rightarrow cx$ the confluence is evident as these rules do not affect the $x$ and do not push letters to the either side of it. For the rest of the rewriting pairs see Appendix A.

Thus, we can establish a normal form for every element in $Q$.

Let $F_n$ be the semigroup defined as follows:

\begin{itemize}
\item It is generated by $a,b,c,d,e,x$ and a zero;
\item We have $xb = cx, ac = ca, ae=ea, ec=ce, xca = xe, aex = ax$ in $F_n$;
\item We also have $a = a^{2n+1}, b = b^{2n+1}, c=c^{2n+1}, e=e^{2n+1}$;
\item If  $s(w)>n$ in $w$, then $w$ represents the zero.
\end{itemize}

We claim that this is a finite semigroup and that the map $f_n$ which is a restriction of the natural homomorphism from $Q^{0}$ to $F_n$ onto $H_n$ satisfies the criteria of Definition \ref{def_lef}.

Consider an arbitrary word $w$. Recall that we can write it as  $m_0 y_0 m_1 y_1 \cdots m_{k-1} y_{k-1} m_k$. We claim that the rewritings corresponding to relations $xb = cx, ac = ca, ae=ea, ec=ce, xca = xe, aex = ax$ or $a = a^{2n+1}, b = b^{2n+1}, c=c^{2n+1}, e=e^{2n+1}$ (denote these $\rho$ for brevity) keep $s(w)$ constant. It follows from the fact that the existence of $x$'s is not affected by elements of $\rho$, while for $b$'s the only way to completely remove its power is to change $xb$ into $cx$, which retains the relative position of the corresponding $y_i$.

Additionally, due to commutativity and periodicity of $a,c,e$, there are at most $(2n+1)^3$ distinct forms of $m_i$. Similarly, there are only $2 \cdot (2n+1)$ distinct forms of $y_i$.

Furthermore, it is evident that for any two words $u,v$ we have $s(uv) \ge \max (s(u) , s(v))$ as concatenation does not decrease this value (only the bordering powers of $b$ can become a single part in the partition of $uv$).

This means that the words with $s(w) > n$ form an ideal in $\{a,b,c,e,x\}^*$, while all the other words correspond to a finite number of elements of $F_n$. Thus, $F_n$ is finite.

In order to see that $f_n$ separates the elements of $H_n$, consider the following rewriting system corresponding to the given relations of $F_n$.

\begin{enumerate}
\item $xb \rightarrow cx, ca \rightarrow ac, ea \rightarrow ae, ec \rightarrow ce$;

\item $a^{2n+1} \rightarrow a, b^{2n+1} \rightarrow b, c^{2n+1} \rightarrow c, e^{2n+1} \rightarrow e$;

\item $x a^\alpha c^\beta \rightarrow x a^{\alpha - \beta} e^\beta$, $0< \beta \le \alpha  \le 2n$;
\item $x a^\alpha c^\beta \rightarrow x c^{\beta - \alpha} e^\alpha$, $0< \alpha < \beta  \le 2n$;
\item $x a^\alpha c^\beta e^\gamma \rightarrow x c^{2n+\beta -\alpha} e^{\alpha+\gamma -2n}$,  $0< \alpha, \gamma \le 2n$, $0 \le \beta \le 2n$, $\alpha+\gamma > 2n, \gamma+\beta \le 2n$;
\item $x a^\alpha c^\beta e^\gamma \rightarrow x a^{2n+\alpha - \beta} e^{\beta+\gamma -2n}$,  $0< \beta, \gamma \le 2n$,  $0 \le \alpha \le 2n$, $\alpha+\gamma \le 2n, \gamma+\beta > 2n$;
\item $x a^\alpha c^\beta e^\gamma \rightarrow x a^{\alpha - \beta} e^{\beta+\gamma - 2n}$,  $0< \alpha,\beta, \gamma \le 2n$, $\alpha+\gamma, \gamma+\beta > 2n$, $\alpha > \beta$;
\item $x a^\alpha c^\beta e^\gamma \rightarrow x e^{\beta+\gamma - 2n}$,  $0< \alpha,\beta, \gamma \le 2n$, $\alpha+\gamma, \gamma+\beta > 2n$,  $\alpha = \beta$;
\item $x a^\alpha c^\beta e^\gamma \rightarrow x c^{\beta - \alpha} e^{\alpha+\gamma -2n}$,  $0< \alpha,\beta, \gamma \le 2n$, $\alpha+\gamma, \gamma+\beta > 2n$, $\beta > \alpha$;

\item $ a^\alpha c^\beta e^\gamma x  \rightarrow a^\alpha c^\beta x $, $0<\alpha, \gamma  \le 2n$, $0 \le \beta \le 2n$;

\item $x e^\gamma x \rightarrow xex$, $1< \gamma  \le 2n$;

\item $x c^\beta e^\gamma x \rightarrow x c^{\beta} e x$, $0< \beta < n$, $1 <\gamma \le 2n$;
\item $x c^\beta e^\gamma x \rightarrow x a^{2n-\beta} x$, $n\le \beta < 2n$, $0 <\gamma \le 2n$ and $x c^{2n} e^\gamma x \rightarrow x e x$, $0 <\gamma \le 2n$;

\item $x a^\alpha c^\beta e^\gamma x \rightarrow x a^{\alpha - \beta} x$, $0<\alpha - \beta \le n$, $0<  \alpha  \le 2n$, $0\le \beta,\gamma \le 2n$ with $\beta$ and $\gamma$ not simultaneously zero;
\item $x a^\alpha c^\beta e^\gamma x \rightarrow x c^{2n-\alpha+\beta} e x$, $n<\alpha - \beta < 2n$, $0< \alpha  \le 2n$, $0\le \beta, \gamma \le 2n$;

\item $x a^\alpha c^\beta e^\gamma x\rightarrow x e x$, $0 <\alpha = \beta \le 2n$, $0\le \gamma \le 2n$;
\item $x a^\alpha c^\beta e^\gamma x\rightarrow x c^{\beta - \alpha} e x$, $0 < \beta - \alpha < n$, $0< \beta, \alpha  \le 2n$, $0\le \gamma \le 2n$;
\item $x a^\alpha c^\beta e^\gamma x\rightarrow x a^{2n - \beta + \alpha} x$, $n \le \beta - \alpha < 2n$, $0< \beta, \alpha  \le 2n$, $0\le \gamma \le 2n$;
\end{enumerate}

It is evident that this system is noetherian as every rule makes the rewritten word smaller in the aforementioned lexicographical order. It is also locally confluent which can be demonstrated using an argument similar to the one for $Q$, with the supplementary table found in Appendix B.

Furthermore, it is clear that the normal forms of the elements of $H_n$ do not represent zero and they are also irreducible under this rewriting system. Thus, $f_n$ separates elements of $H_n$, which concludes the proof.
\end{proof}

\begin{remark}
The example above simultaneously demonstrates that $Q$ is residually finite, but not residually $\J$-trivial finite.
\end{remark}

\section{LWF semigroups}

We turn to the notion of LWF, aiming to prove that it is properly distinct from being LEF for general semigroups. Consider two closely related examples, $$S = Sg \langle a,b,c,e,x | axb = acx, ac = ca, xca = xe, ea = ae, aex = ax, xexb = bxex \rangle$$ and $$T = Sg \langle a,b,c,d,e,x | axb = acx, ac = cd, xcd = xe, ed = ae, aex = ax, xexb = bxex \rangle,$$ which also resemble the final construction of Section 2. Our goal is to demonstrate their properties, including not being LEF and being $\J$-trivial (note that no previous examples of non-LEF $\J$-trivial structures have been known), then develop additional theory of LWF structures and infer that $S$ and $T$ satisfy the new sufficient condition.

\begin{lemma}\label{lem_clarissues}
Let $$S = Sg \langle a,b,c,e,x | axb = acx, ac = ca, xca = xe, ea = ae, aex = ax, xexb = bxex \rangle$$ and $p,z,q \in \{a,b,c,e,x\}^*$ such that $z$ is non-empty, and $z, pzq$ represent the same element of $S$. Then $p = e^t$ for some $t \ge 0$ and $q$ is empty.
\end{lemma}
\begin{proof}
Since $z$ and $pzq$ represent the same element of $S$, there exists a finite sequence of words $w_0 = z, \ldots, w_n = p z q$ such that neighbouring words are different only by a single application of the defining relations of $S$.  Note that none of the rewritings affect the existence of $x$'s inside of words, which means that $p,q$ must be free of $x$'s. There are two further possibilities.

If $z$ does not contain any $x$'s itself, then the only applicable relations are $ea = ae$ and $ac = ca$ which do not affect the length of the word, which means that $p,q$ must be empty.

Otherwise, $z$ has the form $m_0 x m_1 x \cdots x m_k$ where $k \ge 1$ and $m_i$'s are possibly empty words in $\{a,b,c,e\}$. Denote by $u_i$ the prefix of $w_i$ up to the first occurrence of $x$ (in particular, $u_0 = m_0 x$ and $u_n = p m_0 x$). There are three ways our rewritings can affect them:

\begin{enumerate}
\item We apply a rewriting directly to $u_i$ or $u_{i+1}$ itself (so it must be one of the $axb = acx, ac=ca, ea = ae, aex = ax$);
\item We bring in an extra $b$ via $xexb = bxex$;
\item We remove an extra $b$ via $xexb = bxex$.
\end{enumerate}

Note that none of these rewritings add or remove $a$'s from the word, which means that $u_n$ and $u_0$ have the same number of $a$'s in them.

Now we can apply similar analysis to the suffixes $v_i$ of $w_i$ starting with the last occurrence of $x$ (in particular, $v_0 =x m_k$ and $v_n = x m_k q$). There are three ways our rewritings can affect them:

\begin{enumerate}
\item We apply a rewriting directly to $v_i$ or $v_{i+1}$ itself (so it must be one of the $axb = acx, ac=ca, ea = ae, xca = xe$);
\item We bring in an extra $b$ via $xexb = bxex$;
\item We remove an extra $b$ via $xexb = bxex$.
\end{enumerate}

Note that these rewritings keep  the sum $a_i + e_i$ constant where $a_i$ and $e_i$ are the number of occurrences of the respective letter in $v_i$. In particular, this means that $q$ cannot contain $a$ or $e$ as otherwise $a_n + e_n > a_0 + e_0$. 

Bringing the facts about the prefixes and suffixes together we can conclude that $q$ is an empty word and $p$ is a power of $e$ since the defining relations of $S$ keep constant the difference between the number of occurrences of $a$'s and the sum of occurrences of $b$'s and $c$'s (meaning $p,q$ do not contain $b$'s or $c$'s as they do not contain any $a$'s).
\end{proof}

\begin{lemma}\label{lem_weakj}
    The semigroup $$S = Sg \langle a,b,c,e,x | axb = acx, ac = ca, xca = xe, ea = ae, aex = ax, xexb = bxex \rangle$$ is $\J$-trivial.
\end{lemma}

\begin{proof}
Assume that there exist two elements $s,t$ of $S$ such that $s \neq t$ and $s$ is $\J$-related to $t$. Let $u$ and $v$ be words in $\{a,b,c,e,x \}$ representing $s$ and $t$. By the choice of $s$ and $t$ there exist possibly empty words $l_1, r_1$ and $l_2, r_2$ such that $l_1 u r_1$ represents $t$ and $l_2 v r_2$ represents $s$. In particular, this means that we can rewrite $v$ into $l_1 u r_1$ using the defining relations and similarly we can rewrite $u$ into $l_2 v r_2$. Thus, the words $u$ and $l_2 l_1 u r_1 r_2$  represent the same elements of $S$. By Lemma \ref{lem_clarissues} we have that $r_1 r_2$ is empty and $l_1 l_2 = e^t$ for $t \ge 0$.

If $t=0$, $u$ and $v$ represent the same element of $S$ which contradicts our choice of these words. If $t = 1$, then either $l_1$ or $l_2$ is empty and $u,v$ represent the same element of $S$ which contradicts our choice of these words. If $t > 1$, $e^t m_0 x m_1 x \cdots x m_k = m_0 x m_1 x \cdots x m_k$, we claim that $e m_0 x m_1 x \cdots x m_k = m_0 x m_1 x \cdots x m_k$ as well.

This follows from the fact that if our rewritings allow us to eliminate a non-zero number of $e$'s from the prefix $u_n = e^t m_0 x$, then they allow us to eliminate a single $e$ as well. To see this, consider the possible structures of $m_0$.

Case 1. $m_0$ does not contain $a$'s. In this case it is impossible to eliminate any $e$'s as the only rule to do so without $x$ on the left is $aex = ax$.

Case 2. $m_0$ contains at least one $b$ occurring before some non-$b$ letter of $m_0$.  In this case it is also impossible to eliminate any $e$'s as the only rule to do so without $x$ on the left is $aex = ax$, and $b$'s are not subject to any rewriting rules without $x$, meaning that we can neither push them out of the prefix nor get $e$'s past them to be deleted.

Case 3. $m_0$ has the form $g_0 a g_1 a \cdots a g_q$, $q>0$ where $g_j$ are possibly empty words in $\{c,e\}$. Denote $m_1 x \cdots x m_k$ by $w'$. We have $e u = e g_0 a g_1 a \cdots a g_q x w'$, and we can rewrite it into $e u = e a^{q} x b^p w'$, where $p$ is the number of occurrences of $c$'s in $g_0,\ldots, g_q$ combined using the commutativity of $a$ with both $c$ and $e$ and the rules $aex = ax$ and $acx = axb$. Evidently we can rewrite $e a^{q} x b^p w'$ into $a^{q} e x b^p w'$, delete $e$ and then retrace our previous actions backwards, resulting in $ g_0 a g_1 a \cdots a g_q x w' = u$.

Case 4. $m_0$ has the form $g_0 a g_1 a \cdots a g_q b^p$, $q>0$, $p>0$. If the $e$'s can be eliminated from $e^t m_0 = e^t g_0 a g_1 a \cdots a g_q b^p$, this means that we can remove $b$'s from the end of $e^t m_0$. The only way to this is by commuting $b$'s with $xex$ as no other rewriting rules allow us to remove $b$'s from the prefix. Thus, we must be able to rewrite $u$ into a word of the form  $g_0 a g_1 a \cdots a g_q b^p xex w''$ where $w''$ is an arbitrary word (note that the prefix is the same as with $u$ as the characters inside of it cannot interact with ones outside of it as long as $b$'s are not removed). Now we can rewrite this word into  $g_0 a g_1 a \cdots a g_q xex b^p  w''$, or the word $e u$ into $e g_0 a g_1 a \cdots a g_q xex b^p  w''$. Following Case 3, the word  $e g_0 a g_1 a \cdots a g_q xex b^p  w''$ can be rewritten into $g_0 a g_1 a \cdots a g_q xex b^p  w''$, which can be rewritten back into $u$.

To sum up, we have that $e u$ and $u$ represent the same element of $S$. Since $v$  and $l_1 u$ represent the same element of $S$, and $l_1$ is a power of $e$, this means that $u$ and $v$ represent the same element of $S$ which contradicts our choice of these words.

Thus, all cases are impossible and the initial assumption is incorrect.
\end{proof}

\begin{lemma}\label{lem_strongj}
Consider the semigroup $$T = Sg \langle a,b,c,d,e,x | axb = acx, ac = cd, xcd = xe, ed = ae, aex = ax, xexb = bxex \rangle.$$ We have $rst = s \implies r=1, t=1$ for $r,s,t \in T^1$ and in particular $T$ is $\J$-trivial.
\end{lemma} 

\begin{proof}
Assume that there exist $r,s,t \in T^1$ such that $rst=s$ and either $r \neq 1$ or $t \neq 1$. Let $u,w,v$ be words in $\{a,b,c,d,e,x \}$ representing $r,s,t$ (in particular, $u$ and $v$ are not simultaneously empty). We have a finite sequence of words $w_0 = w, \ldots, w_n = u w v$ such that neighbouring words are different only by a single application of the defining relations of $T$. Note that none of the rewritings affect the existence of $x$'s inside of words, which means that $u,v$ must be free of $x$'s. There are two further possibilities.

If $w$ does not contain any $x$'s itself, then the only applicable relations are $ac = cd$ and $ae = ed$ which do not affect the length of the word, which means that $u,v$ must be empty, contradicting our choice of these words.

Otherwise, $w$ has the form $m_0 x m_1 x \cdots x m_k$ where $k \ge 1$ and $m_i$'s are possibly empty words in $\{a,b,c,d,e\}$. Denote by $u_i$ the prefix of $w_i$ up to the first occurrence of $x$ (in particular, $u_0 = m_0 x$ and $u_n = u m_0 x$). There are three ways our rewritings can affect them:

\begin{enumerate}
\item We apply a rewriting directly to $u_i$ or $u_{i+1}$ itself (so it must be one of the $axb = acx, ac=cd, ae = ed, aex = ax$);
\item We bring in an extra $b$ via $xexb = bxex$;
\item We remove an extra $b$ via $xexb = bxex$.
\end{enumerate}

Note that these rewritings preserve the total numbers of $a$'s and $d$'s, which means that $u_n$ and $u_0$ have the same total number of $a$'s and $d$'s in them.

Now we can apply similar analysis to the suffixes $v_i$ of $w_i$ starting with the last occurrence of $x$ (in particular, $v_0 =x m_k$ and $v_n = x m_k v$). There are three ways our rewritings can affect them:

\begin{enumerate}
\item We apply rewriting directly to $v_i$ or $v_{i+1}$ itself (so it must be one of the $axb = acx, ac=cd, ae = ed, xcd = xe$);
\item We bring in an extra $b$ via $xexb = bxex$;
\item We remove an extra $b$ via $xexb = bxex$.
\end{enumerate}

Note that these rewritings keep the sum $a_i + d_i + e_i$ constant, where $a_i,d_i$ and $e_i$ are the number of occurrences of the respective letter in $v_i$. In particular, this means that $v$ cannot contain $a,d$ or $e$ as otherwise $a_n +d_n + e_n > a_0 + d_0 + e_0$.

However, the difference between total numbers of $a$'s and $d$'s and the total number of $b$'s and $c$'s is a constant under all of the described rewritings, which means that $v$ cannot contain $b$'s or $c$'s either as $u$ does not contain $a$'s or $d$'s. This also means that $u$ cannot contain $b$'s or $c$'s, that is it is a power of $e$. Denote $u = e^p$. By our assumption $p >0$.

Consider the possible structures of $m_0$.

Case 1. $m_0$ does not contain $a$'s or $d$'s. In this case it is impossible to eliminate any $e$'s as the only rule to do so without $x$ on the left is $aex = ax$ and the only possibility to obtain an $a$ where there were none is via $ae = ed$ and $ac = cd$.

Case 2. $m_0$ starts with $b$. Let us say that a word in the alphabet $\{a,b,c,d,e\}$ has a $b$-form if it can be written as $b m' x w'$ where $m'$ does not contain $x$'s but contains at least one $a$ or $d$ and $w'$ is some word in $\{a,b,c,d,e\}$. We claim that any rewriting of a word in a $b$-form also has a $b$-form. It follows from the fact that there are no relations allowing to push $b$ to the right of anything but $x$, and here the first $b$ and $x$ are separated, as well as the fact that the total number of $a$'s and $d$'s in $m'$ is preserved under all rewritings.

Additionally, if a word in the alphabet $\{a,b,c,d,e\}$ is a product of $e^p$ and a word in a $b$-form, any rewritings of this word would also have the same structure. To see that, note that no rewritings of $e^p b m' x w'$ can change the prefix of word from being $e^p b$, meaning that they affect $b m' x w'$, which keep its $b$-form as noted above.

Since in this case $w$ has $b$-form and $e^p w$ is a product of $e^p$ and word in $b$-form, they cannot be rewritten into each other.

Case 3. $m_0$ starts with $a,c,e$. Let us say that a word in the alphabet $\{a,b,c,d,e\}$ has an $ace$-form if it can be written as $f m' x w'$ where $f \in \{a,c,e\}$, $m'$ does not contain $x$'s but contains at least one $a$ or $d$ if $f \neq a$ and $w'$ is some word in $\{a,b,c,d,e\}$. We claim that any rewriting of a word in an $ace$-form also has an $ace$-form. It follows from the fact that with such a first letter the only relations which could apply to a subword containing it are $axb =acx, ae=ed, ac=cd, aex = ax$ all of which preserve the required properties.

Additionally, if a word in the alphabet $\{a,b,c,d,e\}$ is a product of $e^p$ and a word in an $ace$-form, any rewritings of this word would also have the same structure. To see that, note that no rewritings of $e^p f m' x w'$ can change the prefix of a word from being $e^p f'$ $f' \in \{a,c,e\}$, and rewritings affecting $f m' x w'$ keep its $ace$-form as noted above.

Since in this case $w$ has an $ace$-form, and $e^p w$ is a product of $e^p$ and a word in an $ace$-form, they cannot be rewritten into each other.

Case 4. $m_0$ starts with $d$, i.e. $m_0 = dm$ for a word $m$. We claim that $e^p w$ cannot be rewritten into $w$ in this case as well. We can apply the results of the previous case for $e^{p-1} ew$, with $ew$ being in an $ace$-form, meaning that all rewritings of $e^{p-1} ew$ start with $e$ (if $p > 1$) or with $a,c$ or $e$ (if $p=1$). In all cases, they cannot start with $d$.

Thus, all cases are impossible and the initial assumption is incorrect.

\end{proof}

\begin{proposition}\label{prop_bigjcon}
The semigroups $$S = Sg \langle a,b,c,e,x | axb = acx, ac = ca, xca = xe, ea = ae, aex = ax, xexb = bxex \rangle$$ and $$T = Sg \langle a,b,c,d,e,x | axb = acx, ac = cd, xcd = xe, ea = de, aex = ax, xexb = bxex \rangle$$ are non-LEF.
\end{proposition}

\begin{proof}
To see that $S$ is non-LEF, first note that we have 
$$
xex = xcax = xacx = xaxb \neq bxax.
$$

To demonstrate this, assume the opposite holds. If $bxax = xex$, there exists a finite sequence of words $w_0, \ldots, w_n$ with $w_0 = bxax$, $w_n = xex$ and neighbouring words different only by a single application of the defining relations of $S$. Let $i$ be the maximal index below $n$ such that $w_i$ starts with $b$. Since  $w_{i+1}$ does not start with $b$, we have $w_{i} = bxex u$ and $w_{i+1} = xex b u$ for some word $u$ as there are no other relations which can remove $b$ from the start of a word. This means that $xex b u$ and $xex$ represent the same element of $S$. However, since $S$ is $\J$-trivial, this would mean that $xex b = xex$ (otherwise they are distinct and $\J$-related). This, however, is incorrect as the rewriting rules in $S$ preserve the difference between the number of occurrences of $a$'s and the sum of occurrences of $b$'s and $c$'s. 

Note that for the same reason $xax \neq xex$.

Now we can consider the finite subset $$H = \{a,b,c,e,x,ax, ac, ae, xe,xa,xex,axb,xax,bxax,xexb\}$$ of $S$. Assume that there exist a finite semigroup $F_H$ and a map $f_H$ satisfying the requirements of Definition \ref{def_lef} for $H$. Denote $af_H$ by $a'$, $b f_H$ by $b'$, $c f_H$ by $c'$, $e f_H$ by $e'$ and $x f_H$ by $x'$. By the multiplication rules we have $a'x'b' = a'c'x', a'c' = c'a', x'c'a' = x'e', e'a' = a'e', a'e'x' = a'x', x'e'x'b' = b'x'e'x'$. In particular, we have $x' a' x' b' = x'a'c' x' = x'e'x'$ and $x' a^{\prime n} x' b' = x' a^{\prime n} c' x' = x'c' a^{\prime n}  x' = x' e' a^{\prime n-1}  x' = x' a^{\prime n-1}  x'$ for $n \ge 2$. Note that since $F_H$ is finite there exist finite positive powers $\eta$ and $\kappa$ such that $a^{\prime \eta} = a^{\prime \eta + \kappa}$. This means that $x' a^{\prime \eta} x' = x' a^{\prime \eta + \kappa} x'$. By multiplying both sides of the equation by $b^{\prime \eta}$ on the right we would arrive at $x' e' x' = x' a^{\prime \kappa} x'$. If $\kappa = 1$, this contradicts the injectivity of $f_H$. If $\kappa > 1$, we claim that it follows that all $x' a^{\prime n} x'$, $1 \le n \le \kappa$ commute with $b'$. We will prove it for $n = \kappa - t$ by induction on $t$.

The base. For $t=0$ it follows from the fact that $x' e' x'$ commutes with $b'$.

The step. Assume the statement holds for $t = t_0 -1$ and consider $t = t_0$, $0 < t_0 < \kappa$. We have $b' x' a^{\prime \kappa - t_0 } x' = b' x' a^{\prime \kappa - t_0 +1 } x' b'$, and by the induction hypothesis this is equal to $ x' a^{\prime \kappa - t_0 +1 } x'b' b' = x' a^{\prime \kappa - t_0 } x' b'$.

This is a contradiction as $x'a'x'$ does not commute with $b'$, since by the injectivity the images of $xaxb = xex$ and $bxax$ must be distinct.

Thus, the subset $H$ is non-embeddable and $S$ is non-LEF.

Similarly, $T$ is not LEF. Note that we also have 
$$
xex = xcdx = xacx = xaxb \neq bxax.
$$

To demonstrate this, assume the opposite holds. If $bxax = xex$, there exists a finite sequence of words $w_0, \ldots, w_n$ with $w_0 = bxax$, $w_n = xex$ and neighbouring words different only by a single application of the defining relations of $T$. Let $i$ be the maximal index below $n$ such that $w_i$ starts with $b$. Since  $w_{i+1}$ does not start with $b$, we have $w_{i} = bxex u$ and $w_{i+1} = xex b u$ for some word $u$ as there are no other relations which can remove $b$ from the start of a word. This means that $xex b u$ and $xex$ represent the same element of $T$. However, since $bu$ represents an element of $T$, this contradicts Lemma \ref{lem_strongj}. 

Note that for the same reason $xax \neq xex = xax b$.

Now we can consider the finite subset $$K = \{a,b,c,d,e,x,ax, ac, ae, xe,xex,axb,xax,bxax,xexb\}$$ of $T$. Assume that there exist a finite semigroup $F_K$ and a map $f_K$ satisfying the requirements of Definition \ref{def_lef} for $K$. Denote $af_K$ by $a'$, $b f_K$ by $b'$, $c f_K$ by $c'$, $d f_K$ by $d'$, $e f_K$ by $e'$ and $x f_K$ by $x'$. By the multiplication rules we have $a'x'b' = a'c'x', a'c' = c'd', x'c'd' = x'e', e'd' = a'e', a'e'x' = a'x', x'e'x'b' = b'x'e'x'$. In particular, we have $x' a' x' b' = x'a'c' x' = x'e'x'$ and $x' a^{\prime n} x' b' = x' a^{\prime n} c' x' = x'c' d^{\prime n}  x' = x' e' d^{\prime n-1}  x' = x' a' e' d^{\prime n-2}  x' = \ldots =  x' a^{\prime n-1}  x'$ for $n \ge 2$. Note that since $F_K$ is finite there exist finite positive powers $\eta$ and $\kappa$ such that $a^{\prime \eta} = a^{\prime \eta + \kappa}$. This means that $x' a^{\prime \eta} x' = x' a^{\prime \eta + \kappa} x'$. By multiplying both sides of the equation by $b^{\prime \eta}$ on the right we would arrive at $x' e' x' = x' a^{\prime \kappa} x'$. If $\kappa = 1$, this contradicts the injectivity of $f_K$. If $\kappa > 1$, we claim that it follows that all $x' a^{\prime n} x'$, $1 \le n \le \kappa$ commute with $b'$. We will prove it for $n = \kappa - t$ by induction on $t$.

The base. For $t=0$ it follows from the fact that $x' e' x'$ commutes with $b'$.

The step. Assume the statement holds for $t = t_0 -1$ and consider $t = t_0$, $0 < t_0 < \kappa$. We have $b' x' a^{\prime \kappa - t_0 } x' = b' x' a^{\prime \kappa - t_0 +1 } x' b'$, and by the induction hypothesis this is equal to $ x' a^{\prime \kappa - t_0 +1 } x'b' b' = x' a^{\prime \kappa - t_0 } x' b'$.

This is a contradiction as $x'a'x'$ does not commute with $b'$, since by the injectivity the images of $xaxb = xex$ and $bxax$ must be distinct.

Thus, the subset $H$ is non-embeddable and $T$ is non-LEF.
\end{proof}

An additional notion is required to provide the new sufficient conditions for being LWF.

\begin{definition}
    Let $S$ be a semigroup generated by a set $X$, $H$ a finite subset of $S$ and $W$ a collection of some words, representing elements of $S$. We say that a product $w_1 \cdots w_k$, $k \ge 1$, $w_i \in W$ is {\em pre-accurate} if

    \begin{enumerate}
\item The element represented by $w_1 \cdots w_k$ belongs to $H$;

\item One of the following holds:

\begin{itemize}
\item[a.] $k=1$;
\item[b.] There exists an index $j$, $1 \le j <k$ such that $w_1 \cdots w_j$ is pre-accurate and $w_{j+1} \cdots w_k$ is pre-accurate.
\end{itemize}
\end{enumerate}
\end{definition}

\begin{proposition}\label{prop_lwfcrit}
    Let $S$ be an infinite semigroup generated by a finite set $X$ and $\{S_m\}_{m=1}^{\infty}$ be a sequence of LEF semigroups generated by the same set. For every $n \ge 1$ denote by $W_n$ the union of words $X \cup X^2 \cdots \cup X^n$, denote by $H_n$ the set of elements of $S$ represented by elements of $W_n$ and by $L_n$  the set of pre-accurate words with respect to $H_n$ and $W_n$. If for all $n \ge 1$ there exists $m \ge 1$ such that 

    \begin{enumerate}
        \item $L_n$ represents a finite subset $H'_n$ of $S_m$;
        \item For $u,v \in L_n$ such that $u$ and $v$ represent the same element in $S_m$, they also represent the same element in $S$;
    \end{enumerate}

    then the semigroup $S$ is LWF.
\end{proposition}

\begin{proof}
    Let $H$ be a finite subset of $S$. There evidently exists an index $n$ such that $H \subseteq H_n$.

    Choose $m$ corresponding to $2n$ according to the given assumption and choose $s \in S \setminus H_{2n}$. Denote by $r$ the natural surjective map between the elements of $H'_{2n}$ and $H_{2n}$ which sends an element of $S_m$ to the element of $S$ represented by the same word in $L_{2n}$ (the existence of such a map is guaranteed by Item 2). Let $F$ be a finite semigroup and $f$ be a map approximating the finite subset $H'_{2n}$ in the sense of Definition \ref{def_lef}. Define the map $d: F \rightarrow S$ as follows:

    \begin{align*}
 x'' d & = \begin{cases} (x'') f^{-1} r, & \text{if }x''f^{-1} \in H'_{2n} \text{(i.e. defined)}, \\ s & \text{otherwise}.  \end{cases}
\end{align*}

    We claim that $F$ and $d$ satisfy Definition  \ref{def_lwf} for the semigroup $S$ and the set $H$. It is straightforward to see that $Fd \supseteq H_{2n} \supseteq H_n \supseteq H$ as the pre-image of $f$ is $H'_{2n}$. Consider $x'', y'' \in F$ such that $x'' d, y'' d \in H$. Denote $x'' f^{-1}$ by $x'$ and  $y'' f^{-1}$ by $y'$. We have $x', y' \in H'_{2n}$, meaning that they are represented by pre-accurate words from $L_{2n}$. However, since $x'r, y'r \in H \subseteq H_n$, the product of these words represents an element from $H_{2n}$ in $S$, meaning that it is pre-accurate itself and $x'y'$ also belongs to $H'_{2n}$ with $(x'r) (y'r) = (x'y')r$. Thus, $x'' y'' = (x'y') f$ and $(x'' d) (y''d) = (x' r) (y'r) = (x' y') r = (x'' y'') d $.
\end{proof}

With the result above, we can return to the two original examples.

\begin{proposition}
The semigroup $T  = Sg\langle a,b,c,d,e,x | axb = acx, ac =
cd, xcd = xe, ea = de, aex = ax, xexb = bxex \rangle$ is LWF.
\end{proposition}
\begin{proof}
Consider the sequence of semigroups $\{D_m\}_{m=1}^{\infty}$ defined by presentations $$Sg\langle a,b,c,d,e,x\rangle / I_m,$$ where $I_m$ is the ideal of the free semigroup on $a,b,c,d,e,x$ consisting of all words of length greater than $m$. We claim that this sequence satisfies the criteria of Proposition \ref{prop_lwfcrit}.

Firstly, each of $D_m$ is LEF as they are finite.

Secondly, choose $n \ge 1$ and consider the set $H_n$ of elements of $T$ represented by words of length at most $n$ and the corresponding set of pre-accurate words $L_n$. Set $m = n 2^{6^{n+1}}$.

Note that $L_n$ is finite. To see this, take a pre-accurate word $w$ of length greater than $n$. It must be a product of two non-empty pre-accurate words $w_0$ and $w_1$. Similarly, if $w_i$, $i \in \{0,1\}$, has a length greater than $n$ it must be a product of pre-accurate  words $w_{i0}$ and $w_{i1}$. We can continue this process until all the factors have a length $n$ or less. It should be noted that in the sequence $w, w_{i_0}, w_{i_0 i_1}, \ldots , w_{i_0 i_1 \ldots i_k}$ all of the corresponding elements of $T$ must be distinct by Lemma \ref{lem_strongj}. This means that such decomposition goes down by no more than $|H_{n}| < 6^{n+1}$ levels, and thus the length of $w$ is bounded by $n 2^{6^{n+1}}$.

Since $D_m$ is finite, $L_n$ represents a finite subset of $D_m$, and by our choice of $m$ no two words of $L_n$ represent the same element in $D_m$. This means that $T$ is LWF.
\end{proof}

\begin{proposition}
The semigroup $S  = Sg\langle a,b,c,e,x | axb = acx, ac =
ca, xca = xe, ea = ae, aex = ax, xexb = bxex \rangle$ is LWF.
\end{proposition}
\begin{proof}
Consider the sequence of semigroups $\{S_m\}_{m=1}^{\infty}$ defined by presentations $$Sg\langle a,b,c,e,x|e^m = e\rangle.$$ We claim that this sequence satisfies the criteria of Proposition \ref{prop_lwfcrit}.

Firstly, each of $S_m$ is LEF. To see this, pick two elements $s,t \in S_m$. There exist words $u,v \in \{a,b,c,e,x\}^*$ representing these elements. Denote by $\bar{l}_e(w)$ the length of a word obtained from a given word $w \in \{a,b,c,e,x\}^*$ by removing all $e$'s. The natural homomorphism from $S_m$ onto the finite quotient $S_m / I_m$, where $I_m$ is the ideal consisting of all elements represented by words $w$ with $\bar{l}_e(w) >\bar{l}_e(u)  + \bar{l}_e(v)$, separates elements $s$ and $t$. This is due to the fact that neither of them lie in the ideal, meaning that $S_m$ is residually finite and, as a consequence, LEF.

Secondly, choose $n \ge 1$ and consider the set $H_n$ of elements of $S$ represented by words of length of at most $n$ and the corresponding set of pre-accurate words $L_n$. Set $m = n+1$.

To understand the structure of $L_n$, take a pre-accurate word $w$ of a length greater than $n$. It must be a product of two non-empty pre-accurate words $w_0$ and $w_1$. Similarly, if $w_i$, $i \in \{0,1\}$ has a length greater than $n$ it must be a product of pre-accurate  words $w_{i0}$ and $w_{i1}$. We can continue this process until all the factors have a length $n$ or less. It should be noted that by Lemma \ref{lem_clarissues} for words $p,z,q \in \{a,b,c,e,x\}^*$ with non-empty $z$ it holds that $z,pzq $ represent the same element of $S$ implies that $q$ is empty and $p$ is a power of $e$.

This allows us to prove that $\bar{l}_e(L_n)$ is finite. Consider a pre-accurate word $w$ with $\bar{l}_e(w) > n$. In the sequence $w, w_{i_0}, w_{i_0 i_1}, \ldots , w_{i_0 i_1 \ldots i_k}$ any two elements are equal if and only if the longer one can be obtained from the shorter one by multiplying it by a power of $e$ on the left, meaning that in the sequence $\bar{l}_e(w), \bar{l}_e(w_{i_0}), \bar{l}_e(w_{i_0 i_1}), \ldots , \bar{l}_e(w_{i_0 i_1 \ldots i_k})$ the number of distinctive elements is less than $|H_n|  < 5^{n+1}$. Since $\bar{l}_e (w_{i_0 i_1 \ldots i_j}) = \bar{l}_e (w_{i_0 i_1 \ldots i_j 0}) + \bar{l}_e (w_{i_0 i_1 \ldots i_j 1})$, we have $\bar{l}_e(w) = \bar{l}_e(w_0) +  \bar{l}_e(w_1)  = \bar{l}_e(w_{00}) +  \bar{l}_e(w_{01})  + \bar{l}_e(w_{10}) +  \bar{l}_e(w_{11}) = \ldots \le n (2^{5^{n+1}})$.

Since there is a finite number of elements of $S_m$ represented by words $w$ with a given value of $\bar{l}_e(w)$, this means that $L_n$ represents a finite number of elements of $S_m$. 

We can also demonstrate that a word  $w$ of the form $u e^k v$ where $k > n$ and $u,v \in \{a,b,c,e,x\}^*$ represents an element of $H_n$ if and only if all $u e^z v$, $0 \le z \le k$ represent the same element. Consider the shortest word $w'$ which represents the same element as $w$. There exists a chain of rewritings $w = w_0, \ldots, w_q = w'$ with rules corresponding to the semigroup presentation of $S$.

If $w'$ does not contain any $x$'s, the only applicable relations are $ea = ae$ and $ac = ca$ which do not change the length of the word and cannot lead to $w$ as its length is at least $k>n$.

Write $w = m_0 x m_1 x \cdots x m_t$ and $w' = m'_0 x m'_1 x \cdots x m'_t $ where $t\ge 1$ and $m_i$'s are possibly empty words in the alphabet $\{a,b,c,e\}$. Assume that $m_j$ contains $e^k$ and denote by $r_i$ the subword of $w_i$ situated between the $j$th and $j+1$st occurrence of $x$. Note that $r_q$ contains strictly less than $|w'| \le n$ occurrences of $e$, meaning that we can remove at least two instances of $e$ from $r_0$ and $e^k$ in particular. 

The only relations which allow us to remove $e$'s from $e^k$ are $xca = xe$ and $aex = ax$.  Note that if the former relation is used on the rewriting step $i$, the prefix of $r_j$ for $j > i$ stays as $c$ unless we use $xca \rightarrow xe$ again (there are no other rules which allow to change $xc$ into anything else). Thus, if we were able to remove at least two $e$'s, $aex = ax$ was also used. However, we can move $e$'s  only by $ae = ea$, meaning that in order to apply this relation some $r_0$ must have suffix $y e^k z$ where

\begin{itemize}
\item $y$ or $z$ has at least one $a$;
\item $y, z$ does not contain any $b$'s;
\item $z$ does not contain $c$ with $e$ to the right of it.
\end{itemize}

Otherwise we either would not be able to find $a$ to eliminate a single $e$ (if there is one, we can push it to the end since $a$ commutes with $c$ and $e$) or push $e$ to the right (we cannot get it past $b$'s at all and if we cannot remove $c$'s which require them to be pushed to the right, we cannot get it past those too). However, this suffix allows us to eliminate any number of $e$'s from $e^k$: first, we push $a$ to the rightmost position, then use it to eliminate any $c$'s in $z$ by commutativity with $a$ and $acx \rightarrow axb$, and then freely remove $e$'s using commutativity with $a$ and $aex \rightarrow ax$. Afterwards we can return the word to the ``initial state" with regards to all letters other than $e$. The latter, in particular, means that $uv$ can be obtained by rewritings from $w$, that is it represents the same element of $S$.

Now, consider $u,v \in L_n$ representing the same element of $S_m$. Since they must be obtained from each other via rewritings corresponding to $e^m = e$, they have the forms $u = y_0 e^{\kappa_1} y_1 \cdots e^{\kappa_h}$ and $v = y_0 e^{\kappa'_1} y_1 \cdots e^{\kappa'_h}$, with $h \ge 0$, $\kappa_i, \kappa'_i >0$, $\kappa_i = \kappa'_i \ (\mod n)$ and $y_i$ being non-empty words in $a,b,c,x$. By the observation above, both of them represent the same element of $H_n$, that is the element given by the word $y_0 e^{\kappa''_1} y_1 \cdots e^{\kappa''_h}$ where $0 \le \kappa''_i <n$ and $\kappa_i = \kappa'_i = \kappa''_i \ (\mod n)$.

Thus, Proposition \ref{prop_lwfcrit} applies and $S$ is LWF.
\end{proof}

\newpage

\section*{Appendix A}

This table demonstrates results of rewritings applied to a word composed of two intersecting left-hand parts of the rules. We specify the second rule when it is not immediately recoverable from the form of the word and reserve the first column for extra conditions.

\begin{center}
\begin{tabular}{|c|c|c|c|c|}
\hline
special & $t$ & $t_1$ & $t_2$ & $t_0$ \\
\hline
\hline
\multicolumn{5}{|c|}{first rule $xb \rightarrow cx$}\\
\hline
\hline
- & $   a^\alpha c^\beta e^\gamma x b  $ &  $  a^\alpha c^\beta e^\gamma c x  $ &   $  a^\alpha c^\beta x b  $ &  $ a^\alpha c^{\beta+1} x  $ \\
\hline
- & $x e^{\gamma} x b  $ &  $  x e^{\gamma} c x  $ &   $   x e x b  $ &  $ x c e x    $ \\
\hline
- & $x c^\beta e^\gamma x b$ &  $   x c^\beta e^\gamma x c  x    $ &   $  x  c^\beta e x b  $ &  $ x c^{\beta+1} e x  $ \\
\hline
$\alpha > \beta+1$ & $ x a^\alpha c^\beta e^\gamma x b  $ &  $  x a^\alpha c^{\beta} e^\gamma c x   $ &   $   x a^{\alpha-\beta} x b  $ &  $ x a^{\alpha-\beta-1} x    $ \\
\hline
$\alpha = \beta+1$ & $  x a^\alpha c^\beta  e^\gamma x b  $ &  $    x a^\alpha c^\beta  e^\gamma c  x  $ &   $   x a x b  $ &  $ x e x    $ \\
\hline
$\alpha = \beta$  & $  x a^\alpha c^\beta e^\gamma x b $ &  $   x a^\alpha c^\beta  e^\gamma  c  x  $ &   $   x e x b  $ &  $ x c e x    $ \\
\hline
$\alpha < \beta$  & $  x a^\alpha c^\beta e^\gamma x b $ &  $   x a^\alpha c^\beta  e^\gamma c  x  $ &   $   x c^{\beta-\alpha} e x b  $ &  $ x c^{\beta-\alpha+1}e x    $ \\
\hline
\hline
\multicolumn{5}{|c|}{first rule $ca \rightarrow ac$}\\
\hline
\hline
-& $ eca $ &  $  eac  $ &   $ cea  $ &  $ ace $ \\
\hline
$\alpha \ge \beta$ & $ x a^\alpha c^\beta a  $ &  $  x a^\alpha c^{\beta-1} a c  $ &   $   x a^{\alpha - \beta} e^\beta a    $ &  $ x a^{\alpha-\beta+1} e^{\beta}     $ \\
\hline
$\alpha < \beta$ & $ x a^\alpha c^\beta a  $ &  $  x a^\alpha c^{\beta-1} a c  $ &   $   x c^{\beta-\alpha} e^{\alpha} a    $ &  $ x c^{\beta-\alpha-1} e^{\alpha+1}     $ \\
\hline
- & $  c a^\alpha c^\beta e^\gamma x   $ &  $  a c a^{\alpha-1} c^\beta e^\gamma x  $ &   $ c a^\alpha c^\beta x   $ &  $ a^\alpha c^{\beta+1} x  $ \\
\hline
\hline
\multicolumn{5}{|c|}{first rule $ea \rightarrow ae$}\\
\hline
\hline
- & $  e a^\alpha c^\beta e^\gamma x   $ &  $  a e a^{\alpha-1} c^\beta e^\gamma x  $ &   $ e a^\alpha c^\beta x   $ &  $ a^\alpha c^{\beta} x  $ \\
\hline
\hline
\multicolumn{5}{|c|}{first rule $x a^\alpha c^\beta \rightarrow x a^{\alpha - \beta} e^\beta$, $0< \beta \le \alpha$}\\
\hline
\hline
\multicolumn{5}{|c|}{second rule $x a^\alpha c^{\beta'} \rightarrow x a^{\alpha - \beta'} e^{\beta'}$, $0< \beta' \le \alpha$, $\beta' < \beta$}\\
\hline
- & $ x a^\alpha c^\beta  $ &  $  x  a^{\alpha - \beta} e^\beta   $ &   $   x a^{\alpha - \beta'} e^{\beta'}  c^{\beta-\beta'}    $ &  $ x a^{\alpha - \beta} e^\beta  $ \\
\hline
\multicolumn{5}{|c|}{second rule $x a^\alpha c^{\beta'} \rightarrow x c^{\beta' -\alpha} e^{\alpha}$, $\beta' > \alpha$}\\
\hline
- & $ x a^\alpha c^{\beta'}  $ &  $  x  a^{\alpha - \beta} e^\beta  c^{\beta'-\beta}  $ &   $   x c^{\beta' -\alpha} e^{\alpha}  $ &  $ x  c^{\beta' -\alpha} e^{\alpha} $ \\
\hline
\multicolumn{5}{|c|}{second rule $a^{\alpha'} c^\beta e^\gamma x  \rightarrow a^{\alpha'} c^\beta x $, $0<\alpha' \le \alpha$, $0 \le \gamma$}\\
\hline
$\alpha > \beta$ &  $ x a^\alpha c^{\beta}  e^\gamma x $ &  $  x  a^{\alpha - \beta} e^{\beta+\gamma}  x $ &   $   x a^{\alpha} c^\beta  x  $ &  $ x  a^{\alpha - \beta} x $\\
\hline
$\alpha = \beta$  &  $ x a^\alpha c^{\beta}  e^\gamma x $ &  $  x  a^{\alpha - \beta} e^{\beta+\gamma}  x $ &   $   x a^{\alpha} c^\beta  x  $ &  $ x  e x $\\
\hline
\multicolumn{5}{|c|}{second rule $ x a^{\alpha} c^{\beta'} e^\gamma x  \rightarrow x a^{\alpha - \beta'} x $, $\beta \le \beta' < \alpha$, $0 \le \gamma$}\\
\hline
- &  $ x a^{\alpha} c^{\beta'} e^\gamma x $ &  $  x  a^{\alpha - \beta} e^{\beta} c^{\beta' -\beta} e^{\gamma}  x $ &   $   x a^{\alpha - \beta'}  x  $ &  $ x  a^{\alpha - \beta'} x $\\
\hline
\multicolumn{5}{|c|}{second rule $ x a^{\alpha} c^{\beta'} e^\gamma x  \rightarrow x e x $, $\beta \le \beta' = \alpha$, $0 \le \gamma$}\\
\hline
- &  $ x a^{\alpha} c^{\beta'} e^\gamma x $ &  $  x  a^{\alpha - \beta} e^{\beta} c^{\beta' -\beta} e^{\gamma}  x $ &   $   x e  x  $ &  $ x  e x $\\
\hline
\multicolumn{5}{|c|}{second rule $ x a^{\alpha} c^{\beta'} e^\gamma x  \rightarrow x  c^{\beta'-\alpha} e  x $, $ \beta' > \alpha$, $0 \le \gamma$}\\
\hline
- &  $ x a^{\alpha} c^{\beta'} e^\gamma x $ &  $  x  a^{\alpha - \beta} e^{\beta} c^{\beta' -\beta} e^{\gamma}  x $ &   $   x  c^{\beta'-\alpha} e  x  $ &  $ x   c^{\beta'-\alpha} e x $\\
\hline
\end{tabular}
\end{center}

\begin{center}
\begin{tabular}{|c|c|c|c|c|}
\hline
special & $t$ & $t_1$ & $t_2$ & $t_0$ \\
\hline
\hline
\multicolumn{5}{|c|}{first rule $x a^\alpha c^\beta \rightarrow x c^{\beta - \alpha} e^\alpha$, $0< \alpha < \beta$}\\
\hline
\hline
\multicolumn{5}{|c|}{second rule $x a^\alpha c^{\beta'} \rightarrow x c^{\beta' -\alpha} e^{\alpha}$, $\beta' > \beta$}\\
\hline
- & $ x a^\alpha c^{\beta'}  $ &  $  x  c^{\beta- \alpha} e^\alpha  c^{\beta'-\beta}  $ &   $   x c^{\beta' -\alpha} e^{\alpha}  $ &  $ x  c^{\beta' -\alpha} e^{\alpha} $ \\
\hline
\multicolumn{5}{|c|}{second rule $a^{\alpha'} c^\beta e^\gamma x  \rightarrow a^{\alpha'} c^\beta x $, $0<\alpha' \le \alpha$, $0 \le \gamma$}\\
\hline
- &  $ x a^\alpha c^{\beta}  e^\gamma x $ &  $  x  c^{\beta - \alpha} e^{\alpha+\gamma}  x $ &   $   x a^{\alpha} c^\beta  x  $ &  $ x  c^{\beta - \alpha} e x $\\
\hline
\multicolumn{5}{|c|}{second rule $ x a^{\alpha} c^{\beta'} e^\gamma x  \rightarrow x  c^{\beta'-\alpha} e  x $, $ \beta' \ge \beta$, $0 \le \gamma$}\\
\hline
- &  $ x a^{\alpha} c^{\beta'} e^\gamma x $ &  $  x  c^{\beta - \alpha} e^{\alpha} c^{\beta' -\beta} e^{\gamma}  x $ &   $   x  c^{\beta'-\alpha} e  x  $ &  $ x   c^{\beta'-\alpha} e x $\\
\hline
\hline
\multicolumn{5}{|c|}{first rule $ a^\alpha c^\beta e^\gamma x  \rightarrow a^\alpha c^\beta x $, $0<\alpha$, $0 \le \beta$, $0 < \gamma$}\\
\hline
\hline
\multicolumn{5}{|c|}{second rule $ a^{\alpha'} c^\beta e^\gamma x  \rightarrow a^\alpha c^\beta x $, $\alpha' > \alpha$}\\
\hline
- &  $  a^{\alpha'} c^\beta e^\gamma x $ &  $  a^{\alpha'} c^\beta   x $ &   $  a^{\alpha'} c^\beta x   $ &  $ a^{\alpha'} c^\beta x  $\\
\hline
\multicolumn{5}{|c|}{second rule $x a^{\alpha'} c^\beta e^\gamma x  \rightarrow x a^{\alpha' -\beta} x $, $\alpha' \ge \alpha, \alpha' > \beta$}\\
\hline
- &  $ x a^{\alpha'} c^\beta e^\gamma x $ &  $  x  a^{\alpha'} c^\beta   x $ &   $  x a^{\alpha' -\beta} x   $ &  $ x a^{\alpha' -\beta} x  $\\
\hline
\multicolumn{5}{|c|}{second rule $x a^{\alpha'} c^\beta e^\gamma x  \rightarrow x e x $, $\alpha' \ge \alpha, \alpha' = \beta$}\\
\hline
- &  $ x a^{\alpha'} c^\beta e^\gamma x $ &  $  x  a^{\alpha'} c^\beta   x $ &   $  x e x   $ &  $ x e x  $\\
\hline
\multicolumn{5}{|c|}{second rule $x a^{\alpha'} c^\beta e^\gamma x  \rightarrow x c^{\beta - \alpha'} e x $, $\alpha' \ge \alpha, \alpha' < \beta$}\\
\hline
- &  $ x a^{\alpha'} c^\beta e^\gamma x $ &  $  x  a^{\alpha'} c^\beta   x $ &   $  x c^{\beta - \alpha'} e x   $ &  $ x c^{\beta - \alpha'} e x  $\\
\hline
\end{tabular}
\end{center}

\section*{Appendix B}

This table is organised similarly to the one in Appendix A.

\begin{center}
\begin{tabular}{|c|c|c|c|c|}
\hline
special & $t$ & $t_1$ & $t_2$ & $t_0$ \\
\hline
\hline
\multicolumn{5}{|c|}{first rule $xb \rightarrow cx$}\\
\hline
\hline
- & $xb^{2n+1}$ & $cxb^{2n}$ & $xb$ & $cx$ \\
\hline
$\beta<2n$ & $   a^\alpha c^\beta e^\gamma x b  $ &  $  a^\alpha c^\beta e^\gamma c x  $ &   $  a^\alpha c^\beta x b  $ &  $ a^\alpha c^{\beta+1} x  $ \\
\hline
$\beta=2n$ & $   a^\alpha c^\beta e^\gamma x b  $ &  $  a^\alpha c^\beta e^\gamma c x  $ &   $  a^\alpha c^\beta x b  $ &  $ a^\alpha c x  $ \\
\hline
- & $x e^{\gamma} x b  $ &  $  x e^{\gamma} c x  $ &   $   x e x b  $ &  $ x c e x    $ \\
\hline
$\beta<n-1$ & $x c^\beta e^\gamma x b$ &   $x c^\beta e^\gamma c x$    &   $xc^\beta e x b$  &   $ xc^{\beta+1} e x$\\
\hline
$\beta=n-1$ & $x c^\beta e^\gamma x b$ &   $x c^\beta e^\gamma c x$    &   $xc^\beta e x b$  &   $ x a^n x$\\
\hline
$n\le \beta <2n-1$ & $x c^\beta e^\gamma x b$ &   $x c^\beta e^\gamma c x$    &   $x a^{2n-\beta} x b$  &   $ x a^{2n-\beta-1} x$\\
\hline
$\beta = 2n-1$ & $x c^\beta e^\gamma x b$ &   $x c^\beta e^\gamma c x$    &   $x a x b$  &   $ x e x$\\
\hline
$\beta = 2n$ &  $x c^\beta e^\gamma x b$ &   $x c^\beta e^\gamma c x$    &   $x e x b$  &   $ x c e x$\\
\hline
$1< \alpha-\beta \le n$ & $ x a^\alpha c^\beta e^\gamma x b$ & $ x a^\alpha c^\beta e^\gamma c x $ & $ x a^{\alpha-\beta} x b $ &  $ x a^{\alpha-\beta -1} x $\\
\hline
$1 = \alpha-\beta $ & $ x a^\alpha c^\beta e^\gamma x b$ & $ x a^\alpha c^\beta e^\gamma c x $ & $ x a x b $ &  $ x e x $\\
\hline
$n+1 = \alpha-\beta$ & $ x a^\alpha c^\beta e^\gamma x b$ & $ x a^\alpha c^\beta e^\gamma c x $ & $ x c^{n-1} e x b $ &  $ x a^n x $\\
\hline
$n+1 < \alpha-\beta \le 2n$ & $ x a^\alpha c^\beta e^\gamma x b$ & $ x a^\alpha c^\beta e^\gamma c x $ & $ x c^{2n-\alpha+\beta} e x b $ &  $ x c^{2n-\alpha+\beta+1} e x $\\
\hline
$\alpha = \beta$ & $ x a^\alpha c^\beta e^\gamma x b$ & $ x a^\alpha c^\beta e^\gamma c x $ & $ x e x b $ &  $ x c e x $\\
\hline
$0<\beta-\alpha<n-1$ & $ x a^\alpha c^\beta e^\gamma x b$ & $ x a^\alpha c^\beta e^\gamma c x $ & $ x c^{\beta-\alpha} e x b $ &  $ x  c^{\beta-\alpha+1} e x $\\
\hline
$\beta-\alpha=n-1$ & $ x a^\alpha c^\beta e^\gamma x b$ & $ x a^\alpha c^\beta e^\gamma c x $ & $ x c^{\beta-\alpha} e x b $ &  $ x  a^n x $\\
\hline
$n \le \beta-\alpha < 2n-1$ & $ x a^\alpha c^\beta e^\gamma x b$ & $ x a^\alpha c^\beta e^\gamma c x $ & $ x a^{2n-\beta+\alpha} x b $ &  $ x  a^{2n-\beta+\alpha-1} x $\\
\hline
$\beta - \alpha = 2n-1$ & $ x a^\alpha c^\beta e^\gamma x b$ & $ x a^\alpha c^\beta e^\gamma c x $ & $ x a x b $ &  $ x e x $\\
\hline
\end{tabular}
\end{center}

\begin{center}
\begin{tabular}{|c|c|c|c|c|}
\hline
special & $t$ & $t_1$ & $t_2$ & $t_0$ \\
\hline
\hline
\multicolumn{5}{|c|}{first rule $ca \rightarrow ac$}\\
\hline
\hline
-& $ eca $ &  $  eac  $ &   $ cea  $ &  $ ace $ \\
\hline
-& $ ca^{2n+1} $ &  $  aca^{2n}  $ &   $ ca  $ &  $ ac $ \\
\hline
-& $ c^{2n+1} a $ &  $  c^{2n} a c $ &   $ ca  $ &  $ ac $ \\
\hline
$ \beta \le \alpha <2n$ & $ x a^\alpha c^\beta a  $ &  $  x a^\alpha c^{\beta-1} a c  $ &   $   x a^{\alpha - \beta} e^\beta a    $ &  $ x a^{\alpha-\beta+1} e^{\beta}  $ \\
\hline
$ 1<\beta, \alpha =2n$ & $ x a^\alpha c^\beta a  $ &  $  x a^\alpha c^{\beta-1} a c  $ &   $   x a^{2n - \beta} e^\beta a    $ &  $ x a^{2n-\beta+1} e^{\beta}  $ \\
\hline
$ 1=\beta, \alpha =2n$ & $ x a^\alpha c^\beta a  $ &  $  x a^{2n+1} c  $ &   $   x a^{2n - 1} e a    $ &  $ x e  $ \\
\hline
$\alpha < \beta$ & $ x a^\alpha c^\beta a  $ &  $  x a^\alpha c^{\beta-1} a c  $ &   $   x c^{\beta-\alpha} e^{\alpha} a    $ &  $ x c^{\beta-\alpha-1} e^{\alpha+1}     $ \\
\hline
$\beta<2n$ & $  c a^\alpha c^\beta e^\gamma x   $ &  $  a c a^{\alpha-1} c^\beta e^\gamma c x  $ &   $ c a^\alpha c^\beta x   $ &  $ a^\alpha c^{\beta+1} x  $ \\
\hline
$\beta = 2n$ & $  c a^\alpha c^\beta e^\gamma x   $ &  $  a c a^{\alpha-1} c^\beta e^\gamma c x  $ &   $ c a^\alpha c^\beta x   $ &  $ a^\alpha c x  $ \\
\hline
\hline
\multicolumn{5}{|c|}{first rule $ea \rightarrow ae$}\\
\hline
\hline
-& $ ea^{2n+1} $ &  $  aea^{2n}  $ &   $ ea  $ &  $ ae $ \\
\hline
-& $ e^{2n+1} a $ &  $  e^{2n} a e $ &   $ ea  $ &  $ ae $ \\
\hline
- & $  e a^\alpha c^\beta e^\gamma x   $ &  $  a e a^{\alpha-1} c^\beta e^\gamma x  $ &   $ e a^\alpha c^\beta x   $ &  $ a^\alpha c^{\beta} x  $ \\
\hline
\multicolumn{5}{|c|}{second rule $x a^\alpha c^\beta e^\gamma \rightarrow x c^{2n+\beta -\alpha} e^{\alpha+\gamma -2n}$,  $0< \alpha, \gamma \le 2n$, $0 \le \beta \le 2n$, $\alpha+\gamma > 2n, \gamma+\beta \le 2n$}\\
\hline
$\alpha = 2n$ & $x a^\alpha c^\beta e^\gamma a$ & $x a^{2n} c^\beta e^{\gamma-1} a e$ & $x c^\beta e^\gamma a$ & $x a c^\beta e^\gamma$ \\
\hline
$\alpha < 2n$ & $x a^\alpha c^\beta e^\gamma a$ & $x a^{\alpha} c^\beta e^{\gamma-1} a e$ & $x c^{2n-\alpha+\beta} e^{\alpha+ \gamma-2n} a$ & $x c^{2n-\alpha+\beta-1} e^{\alpha+\gamma+1-2n} $ \\
\hline
\multicolumn{5}{|c|}{second rule $x a^\alpha c^\beta e^\gamma \rightarrow x a^{2n+\alpha -\beta} e^{\beta+\gamma -2n}$,  $0< \beta, \gamma \le 2n$, $0 \le \alpha \le 2n$, $\beta+\gamma > 2n, \gamma+\alpha \le 2n$}\\
\hline
$\beta-\alpha = 1$ and $\alpha+\gamma=2n$ & $x a^\alpha c^\beta e^\gamma a$ & $x a^{\alpha} c^\beta e^{\gamma-1} a e$ & $x a^{2n-1} e a$ & $x e$ \\
\hline
$\beta - \alpha >1$ and $\alpha+\gamma=2n$  & $x a^\alpha c^\beta e^\gamma a$ & $x a^{\alpha} c^\beta e^{\gamma-1} a e$ & $x a^{2n+\alpha -\beta} e^{\beta+\gamma -2n} a$ & $x c^{\beta-\alpha-1} e$ \\
\hline
$\beta - \alpha >1$ and $\alpha+\gamma<2n$ & $x a^\alpha c^\beta e^\gamma a$ & $x a^{\alpha} c^\beta e^{\gamma-1} a e$ & $x a^{2n+\alpha -\beta} e^{\beta+\gamma -2n} a$ & $x a^{2n+\alpha -\beta+1} e^{\beta+\gamma -2n}$ \\
\hline
\multicolumn{5}{|c|}{second rule $x a^\alpha c^\beta e^\gamma \rightarrow x a^{\alpha - \beta} e^{\beta+\gamma - 2n}$,  $0< \alpha,\beta, \gamma \le 2n$, $\alpha+\gamma, \gamma+\beta > 2n$, $\alpha > \beta$}\\
\hline
$\alpha = 2n, \gamma < 2n$ & $x a^\alpha c^\beta e^\gamma a $ & $x a^{\alpha} c^\beta e^{\gamma-1} a e$ & $x a^{2n - \beta} e^{\gamma+\beta - 2n} a$ & $x a^{2n - \beta+1} e^{\gamma+\beta - 2n} $\\
\hline
$\alpha = 2n, \gamma = 2n$ & $x a^\alpha c^\beta e^\gamma a $ & $x a^{\alpha} c^\beta e^{\gamma-1} a e$ & $x a^{2n - \beta} e^{\beta} a$ & $x c^{\beta-1} e$\\
\hline
$\alpha < 2n$ & $x a^\alpha c^\beta e^\gamma a $ & $x a^{\alpha} c^\beta e^{\gamma-1} a e$ & $x a^{\alpha - \beta} e^{\gamma+\beta - 2n} a$ & $x a^{\alpha - \beta+1} e^{\gamma+\beta - 2n} $\\
\hline
\multicolumn{5}{|c|}{second rule $x a^\alpha c^\beta e^\gamma \rightarrow x  e^{\beta+\gamma - 2n}$,  $0< \alpha,\beta, \gamma \le 2n$, $\alpha+\gamma, \gamma+\beta > 2n$, $\alpha = \beta$}\\
\hline
$\alpha = 2n, \gamma = 2n$ & $x a^\alpha c^\beta e^\gamma a $ & $x a^{\alpha} c^\beta e^{\gamma-1} a e$ & $x  e^{2n} a$ & $x c^{2n-1} e$\\
\hline
$\alpha = 2n, \gamma < 2n$ & $x a^\alpha c^\beta e^\gamma a $ & $x a^{\alpha} c^\beta e^{\gamma-1} a e$ & $x  e^{\gamma} a$ & $x a e^\gamma$ \\
\hline
$\alpha < 2n$ & $x a^\alpha c^\beta e^\gamma a $ & $x a^{\alpha} c^\beta e^{\gamma-1} a e$ & $x  e^{\beta+\gamma - 2n} a$ & $x a e^{\beta+\gamma - 2n}$ \\
\hline
\multicolumn{5}{|c|}{second rule $x a^\alpha c^\beta e^\gamma \rightarrow x c^{\beta - \alpha} e^{\alpha+\gamma -2n}$,  $0< \alpha,\beta, \gamma \le 2n$, $\alpha+\gamma, \gamma+\beta > 2n$, $\beta > \alpha$}\\
\hline
$\beta-\alpha = 1$ & $x a^\alpha c^\beta e^\gamma a$ & $x a^{\alpha} c^\beta e^{\gamma-1} a e$ & $x c e^{\alpha+\gamma -2n} a$ & $x  e^{\alpha+1+\gamma -2n}$ \\
\hline
$\beta-\alpha > 1$ & $x a^\alpha c^\beta e^\gamma a$ & $x a^{\alpha} c^\beta e^{\gamma-1} a e$ & $x c^{\beta-\alpha} e^{\alpha+\gamma -2n} a$ & $x c^{\beta-\alpha-1} e^{\alpha+1+\gamma -2n}$ \\
\hline
\hline
\multicolumn{5}{|c|}{first rule $ec \rightarrow ce$}\\
\hline
\hline
-& $ ec^{2n+1} $ &  $  cec^{2n}  $ &   $ ec  $ &  $ ce $ \\
\hline
-& $ e^{2n+1} c $ &  $  e^{2n} c e $ &   $ ec  $ &  $ ce $ \\
\hline
\multicolumn{5}{|c|}{second rule $x a^\alpha c^\beta e^\gamma \rightarrow x c^{2n+\beta -\alpha} e^{\alpha+\gamma -2n}$,  $0< \alpha, \gamma \le 2n$, $0 \le \beta \le 2n$, $\alpha+\gamma > 2n, \gamma+\beta \le 2n$}\\
\hline
$\alpha -\beta =1$ and $\beta+\gamma = 2n$ & $x a^\alpha c^\beta e^\gamma c$ & $x a^{\alpha} c^\beta e^{\gamma-1} c e$ & $x c^{2n-1} e c$ & $x  e$ \\
\hline
$\alpha -\beta >1$ and $\beta+\gamma = 2n$ & $x a^\alpha c^\beta e^\gamma c$ & $x a^{\alpha} c^\beta e^{\gamma-1} c e$ & $x c^{2n-\alpha+\beta} e^{\alpha+\gamma-2n} c$ & $x a^{\alpha - \beta-1} e $ \\
\hline
$\alpha -\beta >1$ and $\beta+\gamma < 2n$ & $x a^\alpha c^\beta e^\gamma c$ & $x a^{\alpha} c^\beta e^{\gamma-1} c e$ & $x c^{2n-\alpha+\beta} e^{\alpha+\gamma-2n} c$ & $x c^{2n-\alpha+\beta+1}  e^{\alpha+\gamma-2n}$ \\
\hline
\end{tabular}
\end{center}

\begin{center}
\begin{tabular}{|c|c|c|c|c|}
\hline
special & $t$ & $t_1$ & $t_2$ & $t_0$ \\
\hline
\hline
\multicolumn{5}{|c|}{first rule $ec \rightarrow ce$}\\
\hline
\hline
\multicolumn{5}{|c|}{second rule $x a^\alpha c^\beta e^\gamma \rightarrow x a^{2n+\alpha -\beta} e^{\beta+\gamma -2n}$,  $0< \beta, \gamma \le 2n$, $0 \le \alpha \le 2n$, $\beta+\gamma > 2n, \gamma+\alpha \le 2n$}\\
\hline
$\beta = 2n$ & $x a^\alpha c^\beta e^\gamma c$ & $x a^{\alpha} c^\beta e^{\gamma-1} c e$ & $x a^\alpha e^\gamma c$  & $x a^\alpha c e^\gamma$ \\ 
\hline
$\beta < 2n$ & $x a^\alpha c^\beta e^\gamma c$ & $x a^{\alpha} c^\beta e^{\gamma-1} c e$ & $x a^{2n+\alpha -\beta} e^{\beta+\gamma -2n} c$  & $x  a^{2n+\alpha -\beta-1} e^{\beta+\gamma+1 -2n}$ \\ 
\hline
\multicolumn{5}{|c|}{second rule $x a^\alpha c^\beta e^\gamma \rightarrow x a^{\alpha - \beta} e^{\beta+\gamma - 2n}$,  $0< \alpha,\beta, \gamma \le 2n$, $\alpha+\gamma, \gamma+\beta > 2n$, $\alpha > \beta$}\\
\hline
$\alpha -\beta =1$ & $x a^\alpha c^\beta e^\gamma c$ &  $x a^{\alpha} c^\beta e^{\gamma-1} c e$  & $x a e^{\beta+\gamma - 2n} c$ &  $x e^{\beta+\gamma - 2n+1} $\\
\hline
$\alpha -\beta >1$ & $x a^\alpha c^\beta e^\gamma c$ & $x a^{\alpha} c^\beta e^{\gamma-1} c e$ & $x a^{\alpha-\beta} e^{\beta+\gamma - 2n} c$ &  $x a^{\alpha-\beta-1} e^{\beta+\gamma - 2n+1}$\\
\hline
\multicolumn{5}{|c|}{second rule $x a^\alpha c^\beta e^\gamma \rightarrow x  e^{\beta+\gamma - 2n}$,  $0< \alpha,\beta, \gamma \le 2n$, $\alpha+\gamma, \gamma+\beta > 2n$, $\alpha = \beta$}\\
\hline
$\beta =2n, \gamma = 2n$ & $x a^\alpha c^\beta e^\gamma c$ & $x a^{\alpha} c^\beta e^{\gamma-1} c e$  & $xe^{2n} c$ & $x a^{2n-1} e$\\
\hline
$\beta =2n, \gamma < 2n$ & $x a^\alpha c^\beta e^\gamma c$ & $x a^{\alpha} c^\beta e^{\gamma-1} c e$  & $xe^{\gamma} c$ & $x c e^{\gamma}$\\
\hline
$\beta<2n$ & $x a^\alpha c^\beta e^\gamma c$ & $x a^{\alpha} c^\beta e^{\gamma-1} c e$  & $xe^{\gamma+\beta-2n} c$ & $x c e^{\gamma+\beta-2n}$\\
\hline
\multicolumn{5}{|c|}{second rule $x a^\alpha c^\beta e^\gamma \rightarrow x c^{\beta - \alpha} e^{\alpha+\gamma -2n}$,  $0< \alpha,\beta, \gamma \le 2n$, $\alpha+\gamma, \gamma+\beta > 2n$, $\beta > \alpha$}\\
\hline
$\beta =2n, \gamma = 2n$ & $x a^\alpha c^\beta e^\gamma c$ & $x a^{\alpha} c^\beta e^{\gamma-1} c e$  & $x c^{2n -\alpha} e^{\alpha} c$ & $x a^{\alpha-1} e $\\
\hline
$\beta =2n, \gamma < 2n$ & $x a^\alpha c^\beta e^\gamma c$ & $x a^{\alpha} c^\beta e^{\gamma-1} c e$  & $x c^{2n -\alpha} e^{\alpha+\gamma-2n} c$ & $x c^{2n -\alpha+1} e^{\alpha+\gamma-2n}$\\
\hline
$\beta<2n$ & $x a^\alpha c^\beta e^\gamma c$ & $x a^{\alpha} c^\beta e^{\gamma-1} c e$  & $xc^{\beta-\alpha} e^{\gamma+\alpha-2n} c$ & $x c^{\beta-\alpha+1} e^{\alpha+\gamma-2n}$\\
\hline
\hline
\multicolumn{5}{|c|}{first rule $a^{2n+1} \rightarrow a$}\\
\hline
\hline
\multicolumn{5}{|c|}{second rule $a^\alpha c^\beta e^\gamma x  \rightarrow a^\alpha c^\beta x $, $0<\alpha, \gamma  \le 2n$, $0 \le \beta \le 2n$}\\
\hline
$0< \alpha'\le \alpha$ & $a^{2n+\alpha'} c^\beta e^\gamma x$ & $a^{\alpha'} c^\beta e^\gamma x$ & $a^{2n+\alpha'} c^\beta x$ & $a^{\alpha'} c^\beta x$\\
\hline
\hline
\multicolumn{5}{|c|}{first rule $c^{2n+1} \rightarrow c$}\\
\hline
\hline
\multicolumn{5}{|c|}{second rule $x a^\alpha c^\beta \rightarrow x a^{\alpha - \beta} e^\beta$, $0< \beta \le \alpha  \le 2n$}\\
\hline
$0< \beta' \le \beta$ & $x a^\alpha c^{\beta'+2n}$ & $x a^\alpha c^{\beta'}$  & $x a^{\alpha - \beta} e^{\beta} c^{\beta'+2n-\beta}$ & $x a^{\alpha-\beta'} e^{\beta'}$ \\
\hline
\multicolumn{5}{|c|}{second rule $x a^\alpha c^\beta \rightarrow x c^{ \beta -\alpha} e^\alpha$, $0<  \alpha < \beta \le 2n$}\\
\hline
$0< \beta' \le \alpha$ & $x a^\alpha c^{\beta'+2n}$ & $x a^\alpha c^{\beta'}$  & $x c^{\beta - \alpha} e^{\alpha} c^{\beta'+2n-\beta}$ & $x a^{\alpha-\beta'} e^{\beta'}$ \\
\hline
$\alpha < \beta' \le \beta$ & $x a^\alpha c^{\beta'+2n}$ & $x a^\alpha c^{\beta'}$  & $x c^{\beta - \alpha} e^{\alpha} c^{\beta'+2n-\beta}$ & $x c^{\beta'-\alpha} e^{\alpha}$ \\
\hline
\hline
\multicolumn{5}{|c|}{first rule $e^{2n+1} \rightarrow e$}\\
\hline
\hline
\multicolumn{5}{|c|}{second rule $x a^\alpha c^\beta e^\gamma \rightarrow x c^{2n+\beta -\alpha} e^{\alpha+\gamma -2n}$,  $0< \alpha, \gamma \le 2n$, $0 \le \beta \le 2n$, $\alpha+\gamma > 2n, \gamma+\beta \le 2n$}\\
\hline
$0 < \gamma' \le \gamma, \alpha+\gamma' \le 2n$ & $x a^\alpha c^\beta e^{\gamma' +2n}$ & $x a^\alpha c^\beta e^{\gamma'}$ & $x c^{2n+\beta -\alpha} e^{\alpha+\gamma'} $ & $x a^{\alpha -\beta} e^{\beta+\gamma'}$ \\
\hline
$0 < \gamma' \le \gamma, \alpha+\gamma' > 2n$ & $x a^\alpha c^\beta e^{\gamma' +2n}$ & $x a^\alpha c^\beta e^{\gamma'}$ & $x c^{2n+\beta -\alpha} e^{\alpha+\gamma'} $ & $x c^{2n+\beta -\alpha} e^{\alpha+\gamma' -2n}$ \\
\hline
\multicolumn{5}{|c|}{second rule $x a^\alpha c^\beta e^\gamma \rightarrow x a^{2n+\alpha -\beta} e^{\beta+\gamma -2n}$,  $0< \beta, \gamma \le 2n$, $0 \le \alpha \le 2n$, $\beta+\gamma > 2n, \gamma+\alpha \le 2n$}\\
\hline
$0 < \gamma' \le \gamma, \beta+\gamma' \le 2n$ & $x a^\alpha c^\beta e^{\gamma' +2n}$ & $x a^\alpha c^\beta e^{\gamma'}$ & $x a^{2n+\alpha -\beta} e^{\beta+\gamma'} $ & $x c^{\beta -\alpha} e^{\alpha+\gamma'}$ \\
\hline
$0 < \gamma' \le \gamma, \beta+\gamma' > 2n$ & $x a^\alpha c^\beta e^{\gamma' +2n}$ & $x a^\alpha c^\beta e^{\gamma'}$ & $x a^{2n+\alpha -\beta} e^{\beta+\gamma'} $ & $x a^{2n +\alpha -\beta} e^{\beta+\gamma' -2n}$ \\
\hline
\multicolumn{5}{|c|}{second rule $x a^\alpha c^\beta e^\gamma \rightarrow x a^{\alpha - \beta} e^{\beta+\gamma - 2n}$,  $0< \alpha,\beta, \gamma \le 2n$, $\alpha+\gamma, \gamma+\beta > 2n$, $\alpha > \beta$}\\
\hline
$0 < \gamma' \le \gamma$ & $x a^\alpha c^\beta e^{\gamma' +2n}$ & $x a^\alpha c^\beta e^{\gamma'}$ & $x a^{\alpha -\beta} e^{\beta+\gamma'} $ &  $x a^{\alpha -\beta} e^{\beta+\gamma'} $ \\
\hline
\multicolumn{5}{|c|}{second rule $x a^\alpha c^\beta e^\gamma \rightarrow x  e^{\beta+\gamma - 2n}$,  $0< \alpha,\beta, \gamma \le 2n$, $\alpha+\gamma, \gamma+\beta > 2n$, $\alpha = \beta$}\\
\hline
$0 < \gamma' \le \gamma$ & $x a^\alpha c^\beta e^{\gamma' +2n}$ & $x a^\alpha c^\beta e^{\gamma'}$ & $x e^{\beta+\gamma'} $ &  $x e^{\beta+\gamma'} $ \\
\hline
\multicolumn{5}{|c|}{second rule $x a^\alpha c^\beta e^\gamma \rightarrow x c^{\beta - \alpha} e^{\alpha+\gamma -2n}$,  $0< \alpha,\beta, \gamma \le 2n$, $\alpha+\gamma, \gamma+\beta > 2n$, $\beta > \alpha$}\\
\hline
$0 < \gamma' \le \gamma$ & $x a^\alpha c^\beta e^{\gamma' +2n}$ & $x a^\alpha c^\beta e^{\gamma'}$ & $x c^{\beta -\alpha} e^{\alpha+\gamma'} $ &  $x c^{\beta -\alpha} e^{\alpha+\gamma'} $ \\
\hline

\end{tabular}
\end{center}

\begin{center}
\begin{tabular}{|c|c|c|c|c|}
\hline
special & $t$ & $t_1$ & $t_2$ & $t_0$ \\
\hline
\hline
\multicolumn{5}{|c|}{first rule $x a^\alpha c^\beta \rightarrow x a^{\alpha - \beta} e^\beta$, $0< \beta \le \alpha  \le 2n$}\\
\hline
\hline
\multicolumn{5}{|c|}{second rule $x a^\alpha c^{\beta'} \rightarrow x a^{\alpha - \beta'} e^{\beta'}$, $0< \beta' \le \alpha$, $\beta' < \beta$}\\
\hline
- & $ x a^\alpha c^\beta  $ &  $  x  a^{\alpha - \beta} e^\beta   $ &   $   x a^{\alpha - \beta'} e^{\beta'}  c^{\beta-\beta'}    $ &  $ x a^{\alpha - \beta} e^\beta  $ \\
\hline
\multicolumn{5}{|c|}{second rule $x a^\alpha c^{\beta'} \rightarrow x c^{\beta' -\alpha} e^{\alpha}$, $2n \ge \beta' > \alpha$}\\
\hline
- & $ x a^\alpha c^{\beta'}  $ &  $  x  a^{\alpha - \beta} e^\beta c^{\beta'-\beta}  $ &   $   x c^{\beta'-\alpha} e^{\alpha} $ &  $ x c^{\beta'-\alpha} e^{\alpha}  $ \\
\hline
\multicolumn{5}{|c|}{second rule $x a^\alpha c^{\beta'} e^\gamma \rightarrow x c^{2n+\beta' -\alpha} e^{\alpha+\gamma-2n}$, $\gamma+\alpha > 2n$, $\beta' \ge \beta$, $\gamma+\beta'<2n$}\\
\hline
- & $ x a^\alpha c^{\beta'} e^\gamma  $ &  $  x  a^{\alpha - \beta} e^\beta c^{\beta'-\beta} e^\gamma $ &   $   x c^{2n+\beta' -\alpha} e^{\alpha+\gamma-2n} $ &  $ x c^{2n+\beta' -\alpha} e^{\alpha+\gamma-2n}  $ \\
\hline
\multicolumn{5}{|c|}{second rule $x a^\alpha c^{\beta'} e^\gamma \rightarrow x a^{2n+\alpha -\beta'} e^{\beta'+\gamma-2n}$, $\gamma+\alpha \le 2n$, $\beta' > \beta$, $\gamma+\beta'>2n$}\\
\hline
- & $ x a^\alpha c^{\beta'} e^\gamma  $ &  $  x  a^{\alpha - \beta} e^\beta c^{\beta'-\beta} e^\gamma $ & $   x a^{2n+\alpha -\beta'} e^{\beta'+\gamma-2n} $ &  $ x a^{2n+\alpha -\beta'} e^{\beta'+\gamma-2n}  $ \\
\hline
\multicolumn{5}{|c|}{second rule $x a^\alpha c^{\beta'} e^\gamma \rightarrow x a^{\alpha -\beta'} e^{\beta'+\gamma-2n}$,  $\alpha>\beta' \ge \beta$, $\gamma+\alpha, \gamma+\beta' > 2n$}\\
\hline
- & $ x a^\alpha c^{\beta'} e^\gamma  $ &  $  x  a^{\alpha - \beta} e^\beta c^{\beta'-\beta} e^\gamma $ & $   x a^{\alpha -\beta'} e^{\beta'+\gamma-2n} $ &  $ x a^{\alpha -\beta'} e^{\beta'+\gamma-2n}  $ \\
\hline
\multicolumn{5}{|c|}{second rule $x a^\alpha c^{\beta'} e^\gamma \rightarrow x e^{\beta'+\gamma-2n}$,  $\alpha=\beta'$, $\gamma+\alpha> 2n$}\\
\hline
- & $ x a^\alpha c^{\beta'} e^\gamma  $ &  $  x  a^{\alpha - \beta} e^\beta c^{\beta'-\beta} e^\gamma $ & $   x e^{\beta'+\gamma-2n} $ &  $ x e^{\beta'+\gamma-2n}  $ \\
\hline
\multicolumn{5}{|c|}{second rule $x a^\alpha c^{\beta'} e^\gamma \rightarrow x c^{\beta' - \alpha} e^{\alpha'+\gamma-2n}$,  $\alpha<\beta' \le 2n$, $\gamma+\alpha, \gamma+\beta' > 2n$}\\
\hline
- & $ x a^\alpha c^{\beta'} e^\gamma  $ &  $  x  a^{\alpha - \beta} e^\beta c^{\beta'-\beta} e^\gamma $ & $   x c^{\beta' - \alpha} e^{\alpha'+\gamma-2n} $ &  $ x c^{\beta' - \alpha} e^{\alpha'+\gamma-2n}  $ \\
\hline
\multicolumn{5}{|c|}{second rule $a^{\alpha'} c^\beta e^\gamma x  \rightarrow a^{\alpha'} c^\beta x $, $0<\alpha' \le \alpha$, $0 \le \gamma \le 2n$}\\
\hline
$\alpha=\beta$&  $ x a^\alpha c^{\beta}  e^\gamma x $ &  $  x  e^{\beta+\gamma}  x $ &   $   x a^{\alpha} c^\beta  x  $ &  $ x e x $\\
\hline
$0< \alpha-\beta \le n$&  $ x a^\alpha c^{\beta}  e^\gamma x $ &  $  x  a^{\alpha - \beta} e^{\beta+\gamma}  x $ &   $   x a^{\alpha} c^\beta  x  $ &  $ x  a^{\alpha - \beta} x $\\
\hline
$n< \alpha-\beta < 2n$&  $ x a^\alpha c^{\beta}  e^\gamma x $ &  $  x  a^{\alpha - \beta} e^{\beta+\gamma}  x $ &   $   x a^{\alpha} c^\beta  x  $ &  $ x  c^{2n-\alpha +\beta} x $\\
\hline
\multicolumn{5}{|c|}{second rule $x a^\alpha c^{\beta'} e^\gamma x \rightarrow x a^{\alpha-\beta'} x$, $\beta\le \beta'$, $0<\alpha-\beta'\le n$, $0 \le \gamma \le 2n$}\\
\hline
- &  $ x a^\alpha c^{\beta'}  e^\gamma x $ &  $ x  a^{\alpha - \beta} e^\beta c^{\beta'-\beta} e^\gamma  x $ &   $   x a^{\alpha-\beta'}  x  $ &  $ x  a^{\alpha-\beta'} x $\\
\hline
\multicolumn{5}{|c|}{second rule $x a^\alpha c^{\beta'} e^\gamma x \rightarrow x c^{2n-\alpha+\beta'}ex$, $\beta\le \beta'$, $n<\alpha-\beta'\le 2n$, $0 \le \gamma \le 2n$}\\
\hline
- &  $ x a^\alpha c^{\beta'}  e^\gamma x $ &  $ x  a^{\alpha - \beta} e^\beta c^{\beta'-\beta} e^\gamma  x $ &   $   x c^{2n-\alpha+\beta'}e   x  $ &  $ x   c^{2n-\alpha+\beta'}e x $\\
\hline
\multicolumn{5}{|c|}{second rule $x a^\alpha c^{\beta'} e^\gamma x \rightarrow x e x$, $\alpha=\beta'$, $0 \le \gamma \le 2n$}\\
\hline
- &  $ x a^\alpha c^{\beta'}  e^\gamma x $ &  $ x  a^{\alpha - \beta} e^\beta c^{\beta'-\beta} e^\gamma  x $ &   $   x e   x  $ &  $ x  e x $\\
\hline
\multicolumn{5}{|c|}{second rule $x a^\alpha c^{\beta'} e^\gamma x \rightarrow x c^{\beta'-\alpha} e x$, $0<\beta'-\alpha<n$, $0 \le \gamma \le 2n$}\\
\hline
- &  $ x a^\alpha c^{\beta'}  e^\gamma x $ &  $ x  a^{\alpha - \beta} e^\beta c^{\beta'-\beta} e^\gamma  x $ &  $   x c^{\beta'-\alpha} e   x  $ & $ x c^{\beta'-\alpha} e  x $\\
\hline
\multicolumn{5}{|c|}{second rule $x a^\alpha c^{\beta'} e^\gamma x \rightarrow x a^{2n-\beta'+\alpha} x$, $n'\le\beta'-\alpha<2n$, $0 \le \gamma \le 2n$}\\
\hline
- &  $ x a^\alpha c^{\beta'}  e^\gamma x $ &  $ x  a^{\alpha - \beta} e^\beta c^{\beta'-\beta} e^\gamma  x $ &  $   x  a^{2n-\beta'+\alpha}  x  $ & $ x  a^{2n-\beta'+\alpha} x $\\
\hline
\hline
\multicolumn{5}{|c|}{first rule $x a^\alpha c^\beta \rightarrow x c^{\beta - \alpha} e^\alpha$, $0< \alpha <\beta \le 2n$}\\
\hline
\hline
\multicolumn{5}{|c|}{second rule $x a^\alpha c^{\beta'} \rightarrow x c^{\beta' -\alpha} e^{\alpha}$, $2n \ge \beta' > \beta$}\\
\hline
- & $ x a^\alpha c^{\beta'}  $ &  $  x  c^{\beta - \alpha} e^\alpha c^{\beta'-\beta}  $ &   $   x c^{\beta'-\alpha} e^{\alpha} $ &  $ x c^{\beta'-\alpha} e^{\alpha}  $ \\
\hline
\multicolumn{5}{|c|}{second rule $x a^\alpha c^{\beta'} e^\gamma \rightarrow x a^{2n+\alpha -\beta'} e^{\beta'+\gamma-2n}$, $\gamma+\alpha \le 2n$, $\beta' \ge \beta$, $\gamma+\beta'>2n$}\\
\hline
- & $ x a^\alpha c^{\beta'} e^\gamma $ &  $  x  c^{\beta - \alpha} e^\alpha c^{\beta'-\beta} e^\gamma  $ &   $   x a^{2n+\alpha -\beta'} e^{\beta'+\gamma-2n} $ &  $ x a^{2n+\alpha -\beta'} e^{\beta'+\gamma-2n}  $ \\
\hline
\multicolumn{5}{|c|}{second rule $x a^\alpha c^{\beta'} e^\gamma \rightarrow x c^{\beta' - \alpha} e^{\alpha+\gamma-2n}$,  $\beta\le\beta' \le 2n$, $\gamma+\alpha, \gamma+\beta' > 2n$}\\
\hline
- & $ x a^\alpha c^{\beta'} e^\gamma $ &  $  x  c^{\beta - \alpha} e^\alpha c^{\beta'-\beta} e^\gamma  $ &   $   x c^{\beta' - \alpha} e^{\alpha+\gamma-2n} $ &  $ x c^{\beta' - \alpha} e^{\alpha+\gamma-2n}  $ \\
\hline
\multicolumn{5}{|c|}{second rule $a^{\alpha'} c^\beta e^\gamma x  \rightarrow a^{\alpha'} c^\beta x $, $0<\alpha' \le \alpha$, $0 \le \gamma \le 2n$}\\
\hline
$0<\beta-\alpha<n$ &  $ x a^\alpha c^{\beta}  e^\gamma x $ &  $  x  c^{\beta -\alpha} e^{\alpha+\gamma}  x $ &   $   x a^{\alpha} c^\beta  x  $ &  $ x  c^{\beta - \alpha} e x $\\
\hline
$n \le \beta-\alpha<2n$   &  $ x a^\alpha c^{\beta}  e^\gamma x $ &  $  x  c^{\beta -\alpha} e^{\alpha+\gamma}  x $ &   $   x a^{\alpha} c^\beta  x  $ &  $ x a^{2n- \beta + \alpha} x $\\
\hline

\end{tabular}
\end{center}

\begin{center}
\begin{tabular}{|c|c|c|c|c|}
\hline
special & $t$ & $t_1$ & $t_2$ & $t_0$ \\
\hline
\hline
\multicolumn{5}{|c|}{first rule $x a^\alpha c^\beta \rightarrow x c^{\beta - \alpha} e^\alpha$, $0< \alpha <\beta \le 2n$}\\
\hline
\hline
\multicolumn{5}{|c|}{second rule $x a^\alpha c^{\beta'} e^\gamma x \rightarrow x c^{\beta'-\alpha} e x$, $\beta'\ge \beta$, $0<\beta'-\alpha<n$, $0 \le \gamma \le 2n$}\\
\hline
- &  $ x a^\alpha c^{\beta'}  e^\gamma x $ &  $ x  c^{\beta - \alpha} e^\alpha c^{\beta'-\beta} e^\gamma  x $ &  $   x c^{\beta'-\alpha} e   x  $ & $ x c^{\beta'-\alpha} e  x $\\
\hline
\multicolumn{5}{|c|}{second rule $x a^\alpha c^{\beta'} e^\gamma x \rightarrow x c^{\beta'-\alpha} e x$, $\beta'\ge \beta$, $n\le\beta'-\alpha<2n$, $0 \le \gamma \le 2n$}\\
\hline
- &  $ x a^\alpha c^{\beta'}  e^\gamma x $ &  $ x  c^{\beta - \alpha} e^\alpha c^{\beta'-\beta} e^\gamma  x $ &  $   x a^{2n-\beta'+\alpha} x  $ & $ x a^{2n-\beta'+\alpha} x $\\
\hline
\hline
\multicolumn{5}{|c|}{first rule $x a^\alpha c^\beta e^\gamma \rightarrow x c^{2n+\beta - \alpha} e^{\alpha+\gamma-2n}$, $0< \alpha,\gamma  \le 2n$, $0\le \beta\le 2n, \beta+\gamma \le 2n <\alpha+\gamma$}\\
\hline
\hline
\multicolumn{5}{|c|}{second rule $x a^\alpha c^\beta e^{\gamma'} \rightarrow x c^{2n+\beta - \alpha} e^{\alpha+\gamma'-2n}$, $\gamma<\gamma'\le 2n$, $\beta+\gamma' \le 2n <\alpha+\gamma'$}\\
\hline
- &  $ x a^\alpha c^{\beta}  e^{\gamma'} $ &  $ x c^{2n+\beta - \alpha} e^{\alpha+\gamma'-2n}  $ &  $   x c^{2n+\beta - \alpha} e^{\alpha+\gamma'-2n}   $ & $ x c^{2n+\beta - \alpha} e^{\alpha+\gamma'-2n}  $\\
\hline
\multicolumn{5}{|c|}{second rule $x a^\alpha c^\beta e^{\gamma'} \rightarrow x a^{\alpha-\beta} e^{\beta+\gamma'-2n}$, $\gamma<\gamma'\le2n$, $\beta+\gamma',\alpha+\gamma' > 2n$}\\
\hline
- &  $ x a^\alpha c^{\beta}  e^{\gamma'} $ &  $ x c^{2n+\beta - \alpha} e^{\alpha+\gamma'-2n}  $ &  $   x a^{\alpha-\beta} e^{\beta+\gamma'-2n}   $ & $ x a^{\alpha-\beta} e^{\beta+\gamma'-2n}  $\\
\hline
\multicolumn{5}{|c|}{second rule $a^{\alpha'} c^\beta e^\gamma x  \rightarrow a^{\alpha'} c^\beta x $, $0<\alpha' \le \alpha$}\\
\hline
$0< \alpha -\beta \le n$&  $ x a^\alpha c^{\beta}  e^{\gamma}  x$ &  $ x c^{2n+\beta - \alpha} e^{\alpha+\gamma-2n}  x$ &  $   x a^\alpha  c^\beta x  $ & $ x a^{\alpha-\beta} x  $\\
\hline
$n< \alpha -\beta < 2n$&  $ x a^\alpha c^{\beta}  e^{\gamma}  x$ &  $ x c^{2n+\beta - \alpha} e^{\alpha+\gamma-2n}  x$ &  $   x a^\alpha  c^\beta x  $ & $ x c^{2n+\beta - \alpha} e x  $\\
\hline
\multicolumn{5}{|c|}{second rule $x a^{\alpha} c^\beta e^{\gamma'} x  \rightarrow x a^{\alpha -\beta}x $ , $2n\ge \gamma' \ge \gamma$,  $0< \alpha -\beta \le n$}\\
\hline
- &  $ x a^\alpha c^{\beta}  e^{\gamma'}  x$ &  $ x c^{2n+\beta - \alpha} e^{\alpha+\gamma'-2n}  x$ &  $   x a^{\alpha-\beta} x  $ & $ x a^{\alpha-\beta} x  $\\
\hline
\multicolumn{5}{|c|}{second rule $x a^{\alpha} c^\beta e^{\gamma'} x  \rightarrow x c^{2n+\beta - \alpha} e  x$, $2n\ge \gamma' \ge \gamma$,  $n< \alpha -\beta < 2n$}\\
\hline
- &  $ x a^\alpha c^{\beta}  e^{\gamma'}  x$ &  $ x c^{2n+\beta - \alpha} e^{\alpha+\gamma'-2n}  x$ &  $   x c^{2n+\beta - \alpha} e x  $ & $ x c^{2n+\beta - \alpha} e x $\\
\hline
\hline
\multicolumn{5}{|c|}{first rule $x a^\alpha c^\beta e^\gamma \rightarrow x a^{2n+\alpha - \beta} e^{\beta+\gamma-2n}$, $0< \beta,\gamma  \le 2n$, $0\le \alpha \le 2n, \alpha+\gamma \le 2n <\beta+\gamma$}\\
\hline
\hline
\multicolumn{5}{|c|}{second rule $x a^\alpha c^\beta e^{\gamma'} \rightarrow x a^{2n+\alpha - \beta} e^{\beta+\gamma'-2n}$, $\gamma<\gamma'  \le 2n$, $\alpha+\gamma' \le 2n <\beta+\gamma'$}\\
\hline
- & $  x a^\alpha c^\beta e^{\gamma'}$ & $x a^{2n+\alpha - \beta} e^{\beta+\gamma'-2n}$ & $x a^{2n+\alpha - \beta} e^{\beta+\gamma'-2n}$ & $x a^{2n+\alpha - \beta} e^{\beta+\gamma'-2n}$ \\
\hline
\multicolumn{5}{|c|}{second rule $x a^\alpha c^\beta e^{\gamma'} \rightarrow x c^{\beta-\alpha} e^{\alpha+\gamma'-2n}$, $\gamma<\gamma'\le2n$, $\beta+\gamma',\alpha+\gamma' > 2n$}\\
\hline
- & $  x a^\alpha c^\beta e^{\gamma'}$ & $x a^{2n+\alpha - \beta} e^{\beta+\gamma'-2n}$ & $x c^{\beta-\alpha} e^{\alpha+\gamma'-2n}$ & $x c^{\beta-\alpha} e^{\alpha+\gamma'-2n}$ \\
\hline
\multicolumn{5}{|c|}{second rule $a^{\alpha'} c^\beta e^\gamma x  \rightarrow a^{\alpha'} c^\beta x $, $0<\alpha' \le \alpha$}\\
\hline
$0<\beta-\alpha<n$ &  $ x a^\alpha c^{\beta}  e^\gamma x $ &  $  x a^{2n+\alpha - \beta} e^{\beta+\gamma-2n}  x $ &   $   x a^{\alpha} c^\beta  x  $ &  $ x  c^{\beta - \alpha} e x $\\
\hline
$n \le \beta-\alpha<2n$   &  $ x a^\alpha c^{\beta}  e^\gamma x $ &  $  x  a^{2n+\alpha - \beta} e^{\beta+\gamma-2n}  x $ &   $   x a^{\alpha} c^\beta  x  $ &  $ x a^{2n- \beta + \alpha} x $\\
\hline
\multicolumn{5}{|c|}{second rule $x c^\beta e^{\gamma'} x  \rightarrow x c^{\beta}e x$, $2n\ge \gamma' \ge \gamma$,  $0 < \beta < n$, $\alpha = 0$}\\
\hline
- &  $ x c^{\beta}  e^{\gamma'}  x$ &  $ x  a^{2n - \beta} e^{\beta+\gamma'-2n}  x$ &  $   x c^\beta e x  $ & $ x c^\beta e x  $\\
\hline
\multicolumn{5}{|c|}{second rule $x c^\beta e^{\gamma'} x  \rightarrow x a^{2n-\beta} x$, $2n\ge \gamma' \ge \gamma$,  $n \le \beta < 2n$, $\alpha = 0$}\\
\hline
- &  $ x c^{\beta}  e^{\gamma'}  x$ &  $ x  a^{2n - \beta} e^{\beta+\gamma'-2n}  x$ &  $   x a^{2n-\beta} x  $ & $ x a^{2n-\beta} x $\\
\hline
\multicolumn{5}{|c|}{second rule $x c^\beta e^{\gamma'} x  \rightarrow x e x$, $2n\ge \gamma' \ge \gamma$,  $\beta = 2n$, $\alpha = 0$}\\
\hline
- &  $ x c^{\beta}  e^{\gamma'}  x$ &  $ x  a^{2n - \beta} e^{\beta+\gamma'-2n}  x$ &  $   x e x  $ & $ x e x $\\
\hline
\multicolumn{5}{|c|}{second rule $x a^{\alpha} c^\beta e^{\gamma'} x  \rightarrow x c^{\beta -\alpha}e x$, $2n\ge \gamma' \ge \gamma$,  $0< \beta -\alpha < n$}\\
\hline
- &  $ x a^\alpha c^{\beta}  e^{\gamma'}  x$ &  $ x  a^{2n+\alpha - \beta} e^{\beta+\gamma-2n}  x$ &  $   x c^{\beta -\alpha}e x  $ & $ x c^{\beta -\alpha}e x  $\\
\hline
\multicolumn{5}{|c|}{second rule $x a^{\alpha} c^\beta e^{\gamma'} x  \rightarrow x a^{2n-\beta +\alpha} x$, $2n\ge \gamma' \ge \gamma$,  $n \le \beta -\alpha < 2n$}\\
\hline
- &  $ x a^\alpha c^{\beta}  e^{\gamma'}  x$ &  $ x  a^{2n+\alpha - \beta} e^{\beta+\gamma-2n}  x$ &  $   x a^{2n-\beta +\alpha} x  $ & $ x a^{2n-\beta +\alpha} x  $\\
\hline
\end{tabular}
\end{center}

\begin{center}
\begin{tabular}{|c|c|c|c|c|}
\hline
special & $t$ & $t_1$ & $t_2$ & $t_0$ \\
\hline
\hline
\multicolumn{5}{|c|}{first rule $x a^\alpha c^\beta e^\gamma \rightarrow x a^{\alpha - \beta} e^{\beta+\gamma-2n}$, $0< \alpha,\beta,\gamma  \le 2n$, $\alpha+\gamma, \beta+\gamma> 2n$ , $\alpha > \beta$}\\
\hline
\hline
\multicolumn{5}{|c|}{second rule $x a^\alpha c^\beta e^{\gamma'} \rightarrow x a^{\alpha - \beta} e^{\beta+\gamma'-2n}$, $\gamma < \gamma' \le 2n$}\\
\hline
- & $x a^\alpha c^\beta e^{\gamma'} $ & $x a^{\alpha - \beta} e^{\beta+\gamma'-2n}$ & $x a^{\alpha - \beta} e^{\beta+\gamma'-2n}$ & $x a^{\alpha - \beta} e^{\beta+\gamma'-2n}$\\
\hline
\multicolumn{5}{|c|}{second rule $a^{\alpha'} c^\beta e^\gamma x  \rightarrow a^{\alpha'} c^\beta x $, $0<\alpha' \le \alpha$}\\
\hline
$0< \alpha -\beta \le n$&  $ x a^\alpha c^{\beta}  e^{\gamma}  x$ &  $ x a^{\alpha - \beta} e^{\beta+\gamma-2n} x$ &  $   x a^\alpha  c^\beta x  $ & $ x a^{\alpha-\beta} x  $\\
\hline
$n< \alpha -\beta < 2n$&  $ x a^\alpha c^{\beta}  e^{\gamma}  x$ &  $ x a^{\alpha - \beta} e^{\beta+\gamma-2n} x$ &  $   x a^\alpha  c^\beta x  $ & $ x c^{2n+\beta - \alpha} e x  $\\
\hline
\multicolumn{5}{|c|}{second rule $x a^{\alpha} c^\beta e^{\gamma'} x  \rightarrow x a^{\alpha -\beta}x $ , $2n\ge \gamma' \ge \gamma$,  $0< \alpha -\beta \le n$}\\
\hline
- &  $ x a^\alpha c^{\beta}  e^{\gamma'}  x$ &  $ x a^{\alpha - \beta} e^{\beta+\gamma'-2n}   x$ &  $   x a^{\alpha-\beta} x  $ & $ x a^{\alpha-\beta} x  $\\
\hline
\multicolumn{5}{|c|}{second rule $x a^{\alpha} c^\beta e^{\gamma'} x  \rightarrow x c^{2n+\beta - \alpha} e  x$, $2n\ge \gamma' \ge \gamma$,  $n< \alpha -\beta < 2n$}\\
\hline
- &  $ x a^\alpha c^{\beta}  e^{\gamma'}  x$ &  $ x a^{\alpha - \beta} e^{\beta+\gamma'-2n}   x$ &  $   x c^{2n+\beta - \alpha} e x  $ & $ x c^{2n+\beta - \alpha} e x $\\
\hline
\hline
\multicolumn{5}{|c|}{first rule $x a^\alpha c^\beta e^\gamma \rightarrow x e^{\beta+\gamma-2n}$, $0< \alpha,\beta,\gamma  \le 2n$, $\alpha+\gamma, \beta+\gamma> 2n$ , $\alpha = \beta$}\\
\hline
\hline
\multicolumn{5}{|c|}{second rule $x a^\alpha c^\beta e^{\gamma'} \rightarrow x e^{\beta+\gamma'-2n}$, $\gamma< \gamma' \le 2n$ }\\
\hline
- & $ x a^\alpha c^\beta e^{\gamma'} $ & $x e^{\beta+\gamma'-2n} $  & $x e^{\beta+\gamma'-2n}$ & $x e^{\beta+\gamma'-2n}$ \\
\hline
\multicolumn{5}{|c|}{second rule $a^{\alpha'} c^\beta e^\gamma x  \rightarrow a^{\alpha'} c^\beta x $, $0<\alpha' \le \alpha$}\\
\hline
- & $ x a^\alpha c^\beta e^{\gamma} x $ & $x e^{\beta+\gamma-2n} x$  & $x a^{\alpha} c^\beta x$  & $x e x$ \\
\hline
\multicolumn{5}{|c|}{second rule $ x a^{\alpha} c^\beta e^{\gamma'} x  \rightarrow x e x $, $\gamma \le \gamma' \le 2n$}\\
\hline
- & $ x a^\alpha c^\beta e^{\gamma'} x $ & $x e^{\beta+\gamma'-2n} x$  & $x e x$  & $x e x$ \\
\hline
\hline
\multicolumn{5}{|c|}{first rule $x a^\alpha c^\beta e^\gamma \rightarrow x c^{\beta - \alpha} e^{\alpha+\gamma-2n}$, $0< \alpha,\beta,\gamma  \le 2n$, $\alpha+\gamma, \beta+\gamma> 2n$ , $\alpha < \beta$}\\
\hline
\hline
\multicolumn{5}{|c|}{second rule $x a^\alpha c^\beta e^{\gamma'} \rightarrow x c^{\beta - \alpha} e^{\alpha+\gamma'-2n}$, $\gamma < \gamma' \le 2n$}\\
\hline
- & $x a^\alpha c^\beta e^{\gamma'} $  & $ x c^{\beta - \alpha} e^{\alpha+\gamma'-2n}$ & $ x c^{\beta - \alpha} e^{\alpha+\gamma'-2n}$ & $ x c^{\beta - \alpha} e^{\alpha+\gamma'-2n}$ \\
\hline
\multicolumn{5}{|c|}{second rule $a^{\alpha'} c^\beta e^\gamma x  \rightarrow a^{\alpha'} c^\beta x $, $0<\alpha' \le \alpha$}\\
\hline
$0<\beta-\alpha<n$ &  $ x a^\alpha c^{\beta}  e^\gamma x $ &  $  x c^{\beta - \alpha} e^{\alpha+\gamma-2n}  x $ &   $   x a^{\alpha} c^\beta  x  $ &  $ x  c^{\beta - \alpha} e x $\\
\hline
$n \le \beta-\alpha<2n$   &  $ x a^\alpha c^{\beta}  e^\gamma x $ &  $  x  c^{\beta - \alpha} e^{\alpha+\gamma-2n}  x $ &   $   x a^{\alpha} c^\beta  x  $ &  $ x a^{2n- \beta + \alpha} x $\\
\hline
\multicolumn{5}{|c|}{second rule $x a^{\alpha} c^\beta e^{\gamma'} x  \rightarrow x c^{\beta-\alpha} e x $, $\gamma \le \gamma' \le 2n$, $0<\beta-\alpha <n$}\\
\hline
- &  $ x a^\alpha c^{\beta}  e^{\gamma'} x $ &  $  x c^{\beta - \alpha} e^{\alpha+\gamma'-2n}  x $ &   $   x  c^{\beta-\alpha} e  x  $ &  $ x  c^{\beta - \alpha} e x $\\
\hline
\multicolumn{5}{|c|}{second rule $x a^{\alpha} c^\beta e^{\gamma'} x  \rightarrow x a^{2n-\beta+\alpha} x $, $\gamma \le \gamma' \le 2n$, $n\le\beta-\alpha <2n$}\\
\hline
-  &  $ x a^\alpha c^{\beta}  e^{\gamma'} x $ &  $  x  c^{\beta - \alpha} e^{\alpha+\gamma'-2n}  x $ &   $   x a^{2n-\beta+\alpha}  x  $ &  $ x a^{2n- \beta + \alpha} x $\\
\hline
\hline
\multicolumn{5}{|c|}{first rule $a^{\alpha} c^\beta e^\gamma x  \rightarrow a^{\alpha} c^\beta x $, $0<\alpha,\gamma \le 2n$, $0\le \beta \le 2n$}\\
\hline
\hline
\multicolumn{5}{|c|}{second rule $ a^{\alpha'} c^\beta e^\gamma x  \rightarrow a^\alpha c^\beta x $, $2n \ge \alpha' > \alpha$}\\
\hline
- &  $  a^{\alpha'} c^\beta e^\gamma x $ &  $  a^{\alpha'} c^\beta   x $ &   $  a^{\alpha'} c^\beta x   $ &  $ a^{\alpha'} c^\beta x  $\\
\hline
\multicolumn{5}{|c|}{second rule $x a^{\alpha'} c^\beta e^{\gamma} x  \rightarrow x a^{\alpha' -\beta}x $ , $2n\ge \alpha' \ge \alpha$,  $0< \alpha' -\beta \le n$}\\
\hline
- & $x a^{\alpha'} c^\beta e^{\gamma} x$ & $ x  a^{\alpha'} c^\beta x$ & $x a^{\alpha' -\beta}x$  & $ x a^{\alpha' -\beta}x$ \\
\hline
\multicolumn{5}{|c|}{second rule $x a^{\alpha'} c^\beta e^{\gamma} x  \rightarrow x c^{2n-\alpha' +\beta} ex $ , $2n\ge \alpha' \ge \alpha$,  $n< \alpha' -\beta <2n$}\\
\hline
- & $x a^{\alpha'} c^\beta e^{\gamma} x$ & $ x  a^{\alpha'} c^\beta x$ & $x c^{2n-\alpha' +\beta} e x$  & $ x c^{2n-\alpha' +\beta} e x$ \\
\hline
\multicolumn{5}{|c|}{second rule $x a^{\alpha'} c^\beta e^{\gamma} x  \rightarrow x  ex $ , $2n\ge \alpha' \ge \alpha$,  $ \alpha' = \beta$}\\
\hline
- & $x a^{\alpha'} c^\beta e^{\gamma} x$ & $ x  a^{\alpha'} c^\beta x$ & $x e x$  & $ x e x$ \\
\hline
\multicolumn{5}{|c|}{second rule $x a^{\alpha'} c^\beta e^{\gamma} x  \rightarrow x c^{\beta -\alpha'} ex $ , $2n\ge \alpha' \ge \alpha$,  $0<\beta - \alpha' <n$}\\
\hline
- & $x a^{\alpha'} c^\beta e^{\gamma} x$ & $ x  a^{\alpha'} c^\beta x$ & $x c^{\beta -\alpha'} e x$  & $ x c^{\beta -\alpha'} e x$ \\
\hline
\multicolumn{5}{|c|}{second rule $x a^{\alpha'} c^\beta e^{\gamma} x  \rightarrow x a^{2n-\beta +\alpha'} x $ , $2n\ge \alpha' \ge \alpha$,  $n \le \beta - \alpha' <2n$}\\
\hline
- & $x a^{\alpha'} c^\beta e^{\gamma} x$ & $ x  a^{\alpha'} c^\beta x$ & $x a^{2n-\beta +\alpha'}  e x$  & $ x a^{2n-\beta +\alpha'}  x$ \\
\hline

\end{tabular}
\end{center}

\end{document}